\def\version{\standardversion}

\def\stocversion{2}
\ifdefined\version
\else
  \def\version{\stocversion}
\fi
\ifdefined\final
\else
  \def\final{1}
\fi

\ifdefined\overrideversion
    \def\version{\overrideversion}
\fi

\documentclass[11pt]{article} 
\usepackage{fullpage}
\usepackage{url,graphicx,tabularx,array, amsthm, amsfonts,alltt, amsmath,boxedminipage}
\usepackage{hyperref, dsfont,algorithm,algpseudocode,booktabs}
\usepackage[capitalise]{cleveref}
\usepackage{xstring}
\usepackage{microtype}
\usepackage{comment}

\ifnum\version=\stocversion
    \excludecomment{proof}
    \usepackage{sidecap}
    \sidecaptionvpos{figure}{c}
    
    \usepackage{mathptmx}
    \usepackage{enumitem}
    \setlist{nolistsep}
    \newenvironment{myfigure}[1][]{\begin{SCfigure}}{\end{SCfigure}}
\else
    \newenvironment{myfigure}[1][]{%
        \def\temp{#1}\ifx\temp\empty
            \begin{figure}
        \else
            \begin{figure}[#1]
        \fi
    }{%
        \end{figure}%
    }
\fi

\newtheorem{theorem}{Theorem}
\newtheorem{lemma}[theorem]{Lemma}
\newtheorem{definition}[theorem]{Definition}
\newtheorem{claim}{Claim}
\newenvironment{remark}{\paragraph{Remark:}}{}
\newcommand{\email}[1]{\href{mailto:#1}{#1}}
\newenvironment{claimproof}{%
\begin{proof}[Proof of Claim]%
}{%
\begingroup%
\end{proof}%
\endgroup}

\title{The minimum Euclidean-norm point in a convex polytope: Wolfe's combinatorial algorithm is exponential}

\author{ 
Jes{\'u}s De Loera \\
University of California, Davis\\
\email{deloera@math.ucdavis.edu}
\and
Jamie Haddock\\
University of California, Davis\\
\email{jhaddock@math.ucdavis.edu}
\and
Luis Rademacher\\
University of California, Davis\\
\email{lrademac@ucdavis.edu}
}

\date{}

\DeclareMathOperator{\conv}{conv}

\DeclareMathOperator{\aff}{aff}

\DeclareMathOperator{\argmin}{argmin}
\DeclareMathOperator{\R}{\mathbb{R}}

\newenvironment{proofidea}{\noindent{\textit{Proof idea.}}}{\hfill$\square$\medskip}
\newcommand{\proj}{\operatorname{proj}}

\newcommand{\linspan}[1]{\operatorname{span}\left(#1\right)}
\newcommand{\eps}{\epsilon}
\newcommand{\suchthat}{\mathrel{:}}

\newcommand{\RR}{\mathbb{R}}
\newcommand{\QQ}{\mathbb{Q}}

\newcommand{\norm}[1]{{\lVert#1\rVert}}
\newcommand{\norms}[1]{{\lVert#1\rVert}^2}
\newcommand{\enorm}[1]{\norm{#1}_2}
\newcommand{\enorms}[1]{\lVert#1\rVert^2_2}

\ifnum\final=0  
\newcommand{\lnote}[1]{[{\small Luis: \textbf{#1}}]\marginpar{*}}
\newcommand{\anonnote}[1]{[{\small anon: \textbf{#1}}]\marginpar{*}}
\newcommand{\sidecomment}[1]{\marginpar{\tiny #1}}
\newcommand{\details}[1]{[[#1]]}
\else 
\newcommand{\lnote}[1]{}
\newcommand{\anonnote}[1]{}
\newcommand{\sidecomment}[1]{}
\newcommand{\details}[1]{}
\fi  

\noexpandarg
\DeclareRobustCommand{\ve}[1]{%
\IfSubStr{#1}{_}{%
    \StrCut{#1}{_}\csA\csB \mathbf{\csA}_\csB%
}{%
    \mathbf{#1}}%
}
\newcommand{\ppoint}[1]{\ve{p_{#1}}}
\newcommand{\qpoint}[1]{\ve{q_{#1}}}
\newcommand{\rpoint}[1]{\ve{r_{#1}}}
\newcommand{\spoint}[1]{\ve{s_{#1}}}
\newcommand{\Pset}[1]{P(#1)}
\newcommand{\Oset}[1]{O({#1})}
\newcommand{\Cset}[1]{C({#1})}
\newcommand{\M}[1]{M_{#1}}
\newcommand{\m}[1]{m_{#1}}
\newcommand{\xstar}[1]{\ve{o_{#1}^{*}}}
\newcommand{\oracle}{\mathcal{O}}

\begin{document}

\maketitle

\begin{abstract}
    The complexity of Philip Wolfe's method for the minimum Euclidean-norm point problem over a convex polytope has remained unknown since he proposed the method in 1974.
    The method is important because it is used as a subroutine for one of the most practical algorithms for submodular function minimization.
    We present the first example that Wolfe's method takes exponential time.
    Additionally, we improve previous results to show that linear programming reduces in strongly-polynomial time to the minimum norm point problem over a simplex.
\end{abstract}

\ifnum\version=\stocversion
    \newpage
\fi

The fundamental algorithmic problem we consider here is: Given a convex polytope $P \subset \mathbb{R}^d$, 
to find the point $\ve{x} \in P$ of minimum Euclidean norm, i.e., the 
\emph{closest point to the origin} or what we call its \emph{minimum norm point} for short.
We assume $P$ is presented as the convex hull of finitely many points
$\ve{p_1}, \ve{p_2}, \dotsc, \ve{p_n}$ (not necessarily in convex position). 
We wish to find 
\ifnum\version=\stocversion
\[
\begin{array}{ll@{}ll}
\text{argmin}& \;\;\;\; \|\ve{x}\|_2 & &
\\\text{subject to} & \;\;\;\;\;\; \ve{x} & = \underset{k=1}{\overset{n}{\sum}} \lambda_k \ve{p_k}, &
\\
&\underset{k=1}{\overset{n}{\sum}} \lambda_k & = 1, \quad
\lambda_k  \ge 0,  \text{ for } k = 1, 2, ..., n. &
\end{array}
\]
\else
\[
\begin{array}{ll@{}ll}
\text{argmin}& \;\;\;\; \|\ve{x}\|_2 & &
\\\text{subject to} & \;\;\;\;\;\; \ve{x} & = \underset{k=1}{\overset{n}{\sum}} \lambda_k \ve{p_k}, &
\\&\underset{k=1}{\overset{n}{\sum}} \lambda_k & = 1, &
\\& \;\;\;\;\; \lambda_k & \ge 0,  \text{ for } k = 1, 2, ..., n. &
\end{array}
\]
\fi

Finding the minimum norm point in a polytope is a basic auxiliary step in several algorithms arising 
in many areas of optimization and machine learning; a subroutine for solving the minimum norm point problem can be used to compute the projection of
an arbitrary point to a polytope (indeed, $\argmin_{\ve{x} \in P} \|\ve{x}-\ve{a}\|_2$ is the same as 
$\ve{a} + \argmin_{\ve{y} \in P-\ve{a}} \|\ve{y}\|_2$). The minimum norm problem additionally appears in combinatorial optimization, 
e.g., for the nearest point problem for transportation polytopes \cite{abbk,jota} and
as a vital subroutine in B\'ar\'any and Onn's approximation algorithm to
solve the colorful linear programming problem \cite{barany}. 
One of the most important reasons to study this problem is because the minimum norm problem can be
used as a subroutine for submodular function minimization through projection onto the base polytope,
as proposed by Fujishige \cite{fujishige80}. 
Submodular minimization is useful in machine learning, where applications such as 
large scale learning and vision require efficient and accurate solutions \cite{naga2011}.
The problem also appears in optimal loading of recursive neural networks \cite{chandru}. 
The Fujishige-Wolfe algorithm is currently considered an important practical algorithm in 
applications \cite{Fujishige+Isotani, fuji2006,chakrabarty}.
Furthermore, Fujishige et al. first observed that linear programs may be solved 
by solving the minimum norm point problem \cite{fuji2006}, so this simple geometric problem 
is also relevant to the theory of algorithmic complexity of linear optimization. 

One may ask about the complexity of other closely related problems.  First, it is worth remembering that
$L_p$ norm minimization over a polyhedron is NP-hard for $0\leq p<1$ (see \cite{Geetal2011} and the references therein), while
for $p\geq 1$ the convexity of the norm allows for fast computation.
One can prove that it is NP-hard to find a closest \emph{vertex} of a convex polytope given by inequalities.
The reduction for hardness is to the \emph{directed Hamiltonian path problem}: Given a directed graph $G = (V, A)$
and two distinct vertices $s, t \in V$, one aims to decide whether  $G$ contains a directed Hamiltonian path from $s$ to $t$. It is well-known there is a polytope represented by inequalities whose vertices correspond to  the characteristic vectors of a directed path joining $s$ to $t$ in $G$. See e.g.,  Proposition 2.6 in the book 
\cite{nemhauserwolsey1988} for the details, including the explicit inequality description of this polytope. Finally,
by a change of variable $y_i = 1- x_i$, changing zeros to ones and vice versa, the minimum Euclidean norm vertex becomes 
precisely the ``longest path from $s$ to $t$", solving the directed Hamiltonian path problem.


Since the Euclidean norm is a convex quadratic form, the minimum norm point problem is a special case of convex quadratic optimization problem. 
Indeed, it is well-known that a convex quadratic programming problem can be approximately solved in polynomial time; 
that is, some point $\ve{y}$ within distance $\varepsilon$ of the desired minimizing point $\ve{x}$ may be found 
in polynomial time with respect to $\log \frac1\varepsilon$. This can be done either through several iterative (convergent)
algorithms, such as the Ellipsoid method \cite{kozlovetal} and interior-point method techniques \cite{boyd2004convex}.
Each of these are methods whose complexity depends upon the desired accuracy.
However, an approximate numerical solution is inconvenient when the application 
requires more information, e.g., if we require to know the face that contains the minimum norm point. 
Numeric methods that converge to a solution and require further rounding are not as convenient for this need. 

In this paper, we focus on combinatorial algorithms that rely on the structure of the polytope.
There are several reasons to study the complexity of combinatorial algorithms for the minimum norm problem. 
On the one hand, the minimum norm problem can indeed be solved in strongly-polynomial time for some polytopes; most notably in network-flow and transportation polytopes (see \cite{jota,abbk,Vegh16}, and references therein, for details).
On the other hand, while linear programming reduces to the minimum norm problem, it is unknown whether linear programming can be solved in strongly-polynomial time \cite{smale2000mathematical}, thus the complexity of the minimum norm point problem could also impact the algorithmic efficiency of linear programming and optimization in general. 
For all these reasons it is natural to ask whether a strongly-polynomial time algorithm exists for the minimum norm problem for general polytopes.


\paragraph{ Our contributions:}
\begin{itemize}
\item  In 1974, Philip Wolfe proposed a combinatorial method that can solve the minimum-norm point problem exactly \cite{wolfe,wolfeprime}. 
Since then, the complexity of Wolfe's method was not understood. In Section 1 we present 
our main contribution and give the first example that Wolfe's method has exponential behavior.
This is akin to the well-known Klee-Minty examples showing exponential behavior 
for the simplex method \cite{klee+minty}. Prior work by \cite{chakrabarty} showed that after $t$ iterations, Wolfe’s method returns an $O(1/t)$-approximate solution to the minimum norm point on any polytope. 

\item As we mentioned earlier, an enticing reason to explore the complexity of the minimum norm problem is its 
intimate link to the complexity of linear programming. It is known that linear programming can be polynomially 
reduced to the minimum norm point problem \cite{Fujishige+Isotani}. 
In Section 2, we strengthen earlier results showing that linear optimization is strongly polynomial time reducible to the minimum norm point problem on a simplex.  

\end{itemize}

\ifnum\version=\stocversion
Proofs omitted from this extended abstract can be found in the full version \cite{1710.02608} and attached at the end of this extended abstract.
\fi






\section{Wolfe's method exhibits exponential behavior}

For convenience of the reader and to set up notation we start with a brief description of Wolfe's method.  We will then describe our exponential example in detail, proving the exponential behavior of Wolfe's method.
First, we review two important definitions.
Given a set of points $S \subseteq \RR^d$, we have two minimum-norm points to consider.  
One is the \emph{affine minimizer} which is the point of minimum norm in the affine hull of $S$, $\argmin_{\ve{x} \in \aff(S)} \|\ve{x}\|_2$.
The second is the \emph{convex minimizer} which is the point of minimum norm in the convex hull of $S$, $\argmin_{\ve{x} \in \conv(S)} \|\ve{x}\|_2$.  Note that solving for the convex minimizer of a set of points is exactly the problem we are solving, while solving for the affine minimizer of a set of points is easily computable.

\subsection{A brief review of Wolfe's combinatorial method}\label{Wolfeintro}


Wolfe's method is a combinatorial method for solving the minimum norm point problem over a polytope, $P = \conv(\ve{p_1}, \ve{p_2}, ..., \ve{p_n}) \subset \mathbb{R}^d$, introduced by P.~Wolfe in \cite{wolfe}. 
The method iteratively solves this problem over a sequence of subsets of no more than $d+1$ affinely independent points from $\ve{p_1}, ..., \ve{p_n}$ and it checks to see if the solution to the subproblem is a solution to the problem over $P$ using the following lemma due to Wolfe. We call this \emph{Wolfe's criterion}.

\begin{lemma}[Wolfe's criterion \cite{wolfe}]\label{lem:mnpcheck}\label{wolfec}
Let $P = \conv(\ve{p_1}, \ve{p_2}, ..., \ve{p_n}) \subset \R^d$, then $\ve{x} \in P$ is the minimum norm point in $P$ if and only if 
\ifnum\version=\stocversion
$\ve{x}^T \ve{p_j} \ge \|\ve{x}\|_2^2$ for all  $j \in [n]$.
\else
\[
\ve{x}^T \ve{p_j} \ge \|\ve{x}\|_2^2 \quad \text{for all} \quad j \in [n].
\]
\fi
\end{lemma}

Note that this tells us that if there exists a point $p_j$ so that $x^Tp_j < \enorms{x}$ then $x$ is not the minimum norm point in $P$.  
We say that $p_j$ violates Wolfe's criterion and using this point should decrease the minimum norm point of the current subproblem.

It should be observed that just as Wolfe's criterion is a rule to decide optimality over $\conv(P)$, one has a very similar
rule for deciding optimality over the affine hull, $\aff(P)$. 
\ifnum\version=\stocversion
We state and prove this result below since we do not know of a reference.
\fi


\begin{lemma}[Wolfe's criterion for the affine hull]\label{lem:affineoptimality}
Let $P = \{\ve{p_1},\ve{p_2},...,\ve{p_n}\} \subseteq \RR^d$ be a non-empty finite set of points. Then $\ve{x} \in \aff P$ is the minimum norm point in $\aff P$ iff for all $\ve{p_i} \in P$ we have $\ve{p_i}^T \ve{x} = \enorms{\ve{x}}$.
\end{lemma}
\begin{proof}
($\Leftarrow$) Let $\ve{p}= \sum_{i=1}^n \rho_i \ve{p_i}$ with $\sum_{i=1}^n \rho_i = 1$ be an arbitrary point in $\aff P$ and suppose $\ve{p_i}^T \ve{x} = \enorms{\ve{x}}$ for $i=1,2,...,n$.  
We have 
\[
\ve{p}^T \ve{x} = \underset{i=1}{\overset{n}{\sum}} \rho_i \ve{p_i}^T \ve{x} = \underset{i=1}{\overset{n}{\sum}}\rho_i\enorms{\ve{x}} = \enorms{\ve{x}}.
\]
Then $0 \le \enorms{\ve{p}-\ve{x}} = \enorms{\ve{p}} - 2\ve{p}^T \ve{x} + \enorms{\ve{x}} = \enorms{\ve{p}} - \enorms{\ve{x}}$ and so $\enorms{\ve{x}} \le \enorms{\ve{p}}$.

($\Rightarrow$) Suppose $\ve{x} \in \aff P$ is the minimum norm point in $\aff P$.  Suppose that $\ve{x}^T(\ve{p_i} - \ve{x}) \not= 0$ for some $i \in [n]$.  First, consider the case when $\ve{x}^T(\ve{p_i} - \ve{x}) > 0$ and define $0 < \epsilon < \frac{2\ve{x}^T(\ve{p_i}-\ve{x})}{\enorms{\ve{p_i} - \ve{x}}}.$  Then we have 
\[
\enorms{(1+\epsilon)\ve{x} - \epsilon \ve{p_i}} = \enorms{x + \epsilon(x - p_i)} = \enorms{\ve{x}} - 2\epsilon \ve{x}^T(\ve{p_i}-\ve{x}) +\epsilon^2 \enorms{\ve{p_i}-\ve{x}} < \enorms{\ve{x}}
\]  
since $0 < \epsilon^2 \enorms{\ve p_i - \ve x} < 2\epsilon x^T(\ve p_i - \ve x)$.  This contradicts our assumption that $\ve{x}$ is the minimum norm point in $\aff P$.  The case when $\ve{x}^T(\ve{p_i} -\ve{x}) < 0$ is likewise proved by considering $\enorms{(1-\epsilon)\ve{x} + \epsilon \ve{p_i}}$ with $0< \epsilon < -\frac{2\ve{x}^T(\ve{p_i}-\ve{x})}{\enorms{\ve{p_i} - \ve{x}}}$.  Thus, we have that $\ve{x}^T(\ve{p_i} - \ve{x}) = 0$.
\end{proof}

We say a set of affinely independent points $S$ is a \emph{corral} if the affine minimizer of $S$ 
lies in the relative interior of $\conv{S}$.
Note that singletons are always corrals. 
Carath{\'e}odory's theorem implies that the minimum norm point of $P$ will lie in the convex hull of some corral of points among $\ve{p_1},...,\ve{p_n}$.  
The goal of Wolfe's method is to search for a corral containing the (unique) minimizing point.

The pseudo-code in Method \ref{alg:wolfe} below presents the iterations of Wolfe's method. It is worth noticing that some steps of the method can be implemented in more than one way and Wolfe proved that all of them lead to a correct algorithm (for example, the choice of the initial point in line \ref{initialrule}). We therefore use the word \emph{method} to encompass all these variations and we discuss specific choices when they are relevant to our analysis of the method.

\begin{algorithm}
\floatname{algorithm}{Method}
\caption{Wolfe's Method \cite{wolfe}}\label{alg:wolfe}
\begin{algorithmic}[1]
\Procedure{Wolfe}{$\ve{p_1}, \ve{p_2}, ..., \ve{p_n}$}
\State Initialize $\ve{x} = \ve{p_i}$ for some $i \in [n]$, initial corral $C = \{\ve{p_i}\}$, $I=\{i\}$, $\ve{\lambda} = \ve{e_i}$, $\ve{\alpha} = \ve{0}$.\label{initialrule}
\While{$\ve{x} \not= \ve{0}$ and there exists $\ve{p_j}$ with $\ve{x}^T\ve{p_j} < \|\ve{x}\|_2^2$}\label{stoppingcriterion}
\State Add $\ve{p_j}$ to the potential corral: $C = C \cup \{\ve{p_j}\}$, $I = I \cup \{j\}$. \label{addrule}
\State Find the affine minimizer of $C$, $\ve{y} = \argmin_{\ve{y} \in \aff(C)} \|\ve{y}\|_2$, and the affine coefficients, $\alpha$.
\While{$\ve{y}$ is not a strict convex combination of points in $C$; $\alpha_i \le 0$ for some $i \in I$}
\State Find $\ve{z}$, closest point to $\ve{y}$ on $[\ve{x},\ve{y}] \cap \conv(C)$; $\ve{z} = \theta \ve{y} + (1-\theta)\ve{x}$,
$\theta = \min_{i \in I : \alpha_i \le 0} \frac{\lambda_i}{\lambda_i - \alpha_i}$.
\State Select $\ve{p_i} \in \{\ve{p_j} \in C : \theta \alpha_j + (1-\theta)\lambda_j = 0\}$.
\State Remove this point from $C$; $C = C - \{\ve{p_i}\}$, $I = I - \{i\}$, $\alpha_i = 0$, $\lambda_i = 0$.
\State Update $\ve{x} = \ve{z}$ and the convex coefficients, $\lambda$, of $\ve{x}$ for $C$; solve $\ve{x} = \sum_{\ve{p_i} \in C} \lambda_i \ve{p_i}$ for $\lambda$.
\State Find the affine minimizer of $C$, $\ve{y} = \argmin_{\ve{y} \in \aff(C)} \|\ve{y}\|_2$ and the affine coefficients, $\alpha$.
\EndWhile
\State Update $\ve{x} = \ve{y}$ and $\lambda = \alpha$.
\EndWhile
\State \textbf{Return} $\ve{x}$.
\EndProcedure
\end{algorithmic}
\end{algorithm}


The subset of points being considered as the \emph{potential corral} is maintained in the set $C$.  Iterations of the outer-loop, where points are added to $C$, are called \emph{major cycles} and iterations of the inner-loop, where points are removed from $C$, are called \emph{minor cycles}.  
The potential corral, $C$, is named so because at the beginning of a major cycle it is guaranteed to be a corral, while within the minor cycles it may or may not be a corral.  
Intuitively, a major cycle of Wolfe's method inserts an \emph{improving point} 
which violates 
Wolfe's criterion ($\ve{p_j}$ so that $\ve{x}^T \ve{p_j} < \|\ve{x}\|_2^2$) into $C$, then the minor cycles remove points until $C$ is a corral, and this process is repeated until no points are improving and $C$ is guaranteed to be a corral containing the minimizer.

It can be shown that this method terminates because the norm of the convex minimizer of the corrals visited monotonically decreases and thus, no corral is visited twice \cite{wolfe}. 
\ifnum\version=\stocversion
\else
Like \cite{chakrabarty}, we sketch the argument in \cite{wolfe}.  
One may see that the norm monotonically decreases by noting that the convex minimizer over the polytope may result from one of two updates to $\ve{x}$, either at the end of a major cycle or at the end of a minor cycle.  
Let $C$ be the corral at the beginning of a major cycle (line 3 of Method \ref{alg:wolfe}) and let $\ve{x}$ be the current minimizer, then the affine minimizer $\ve{y}$ has norm strictly less than that of $\ve{x}$ by Lemma \ref{lem:affineoptimality}, uniqueness of the affine minimizer and the fact that $\ve{p_i}^T\ve{x} < \enorms{\ve{x}}$ where $\ve{p_i}$ is the added point.  
Now, either $\ve{x}$ is updated to $\ve{y}$ or a minor cycle begins.  
Let $S$ be the potential corral at the beginning of a minor cycle (line 6 of \ref{alg:wolfe}), let $\ve{x}$ be the current convex combination of points of $S$ and let $\ve{y}$ be the affine minimizer of $S$.  
Note that $\ve{z}$ is a proper convex combination of $\ve{x}$ and $\ve{y}$ and since $\enorm{\ve{y}} < \enorm{\ve{x}}$, we have $\enorm{\ve{z}} < \enorm{\ve{x}}$.  
Thus, we see that every update of $\ve{x}$ decreases its norm.  
Note that the number of minor cycles within any major cycle is bounded by $d+1$, where $d$ is the dimension of the space.  
Thus, the total number of iterations is bounded by the number of corrals visited multiplied by $d+1$.  
\fi
It is nevertheless not clear how the number of corrals grows, beyond the bound of $\sum_{i=1}^{d+1} \binom ni$.

Within the method, there are two moments at which one may choose which points to add to the potential corral.  
Observe that at line \ref{initialrule} of the pseudocode, one may choose which initial point to add to the potential corral.  
In this paper we will only consider one \emph{initial rule}, which is to initialize with the point of minimum norm. 
Observe that at line \ref{addrule} of the pseudocode, there are several potential choices of which point to add to the potential corral. 
Two important examples of \emph{insertion rules} are, first, the \emph{minnorm rule} which dictates that one chooses, out of the improving points for the potential corral, to add the point $\ve{p_j}$ of minimum norm.  
Second, the \emph{linopt rule} dictates that one chooses, out of the improving points for the potential corral, to add the point $\ve{p_j}$ minimizing $\ve{x}^T \ve{p_j}$. Notice that insertion rules are to Wolfe's method what \emph{pivot rules} are to the Simplex Method (see \cite{Terlaky+Shuzhong1993} for a summary). 
\lnote{in view of my footnote later, a discussion of removal rule may be relevant here.}

As with pivot rules, there are advantages and disadvantages of insertion rules. 
For example, the minnorm rule has the advantage that its implementation only requires an initial ordering of the points, then in each iteration it need only to search for an improving point in order of increasing norm and to add the first found.  
However, the linopt insertion rule has the advantage that, if the polytope is given in H-representation (intersection of halfspaces) rather than V-representation (convex hull of points), one may still perform Wolfe's method by using linear programming to find $\ve{p_j}$ minimizing $ \ve{x}^T \ve{p_j}$ over the polytope. 
In other words, Wolfe's method does not need to have the list of vertices explicitly given, but suffices to have a linear programming oracle that provides the new vertex to be inserted.
This feature of Wolfe's method means that each iteration can be implemented efficiently even for certain polyhedra having too many vertices and facets: specifically, over zonotopes (presented as a Minkowski sum of segments) \cite{fuji2006} and over the base polyhedron of a submodular function \cite{fujishige80}.

\ifnum\version=\stocversion
We present examples that show that the optimal choice of insertion rule depends on the input data. 
In the full version \cite{1710.02608} (and appended at the end here) of this extended abstract we present a simple example where the minnorm rule outperforms the linopt rule.
That is, the minnorm pivot rule is not in obvious disadvantage to the linopt rule.
In \cref{sec:explowerbound}, we present a family of examples where the minnorm rule takes exponential time, while we expect the linopt rule to take polynomial time in this family.
\else
Now we present examples that show that the optimal choice of insertion rule depends on the input data. 
We first present a simple example where the minnorm rule outperforms the linopt rule.
That is, the minnorm insertion rule is not in obvious disadvantage to the linopt rule.
In \cref{sec:explowerbound}, we present a family of examples where the minnorm rule takes exponential time, while we expect the linopt rule to take polynomial time in this family.


\begin{myfigure}[ht]
\centering
		\includegraphics[width=.4\columnwidth]{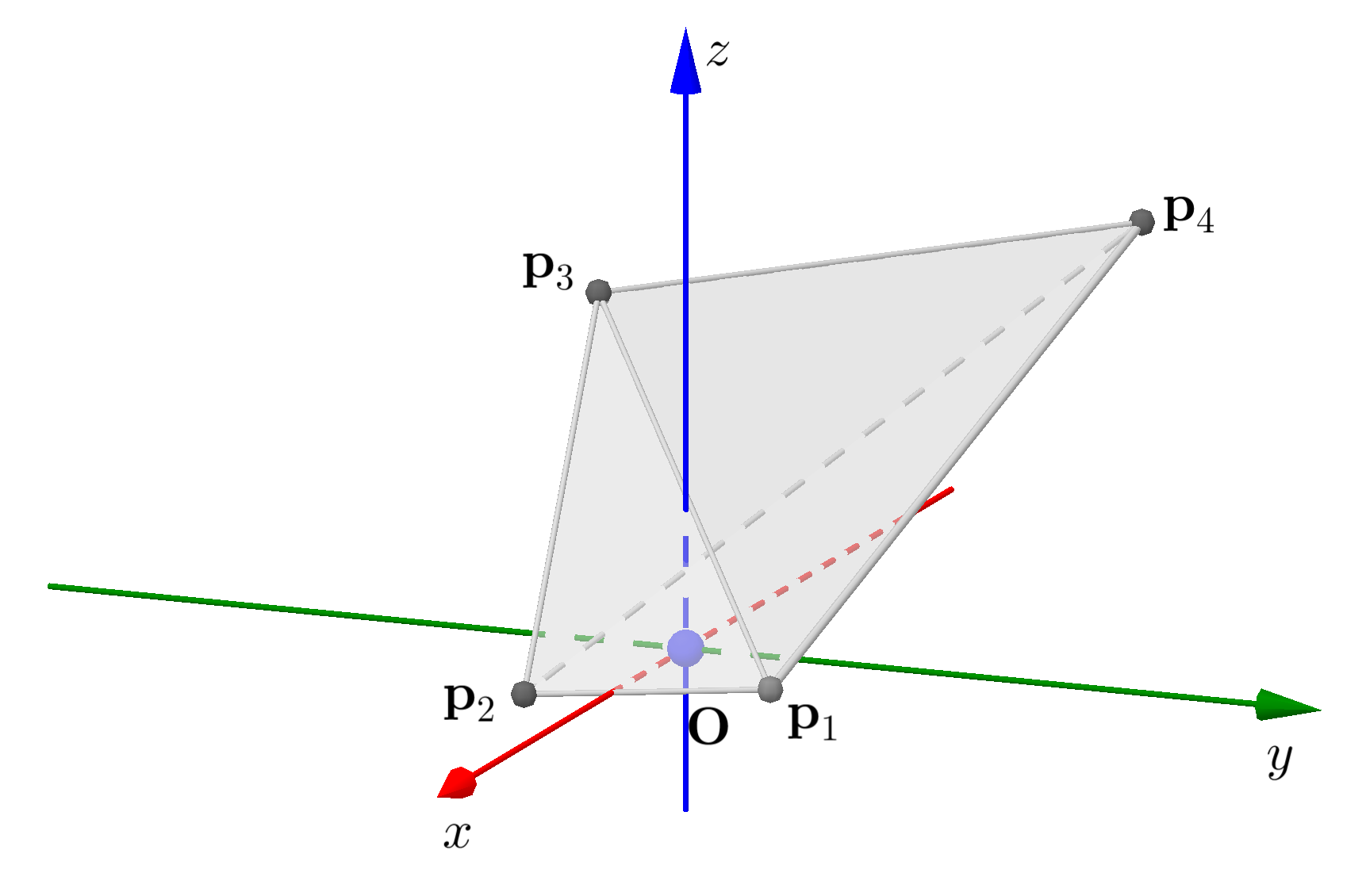}
	\caption{The simplex $P = \conv\{\ve{p_1},\ve{p_2},\ve{p_3},\ve{p_4}\} \subset \mathbb{R}^3$ where $\ve{p_1} = (0.8,0.9,0), \ve{p_2} = (1.5,-0.5,0), \ve{p_3} = (-1,-1,2)$ and $\ve{p_4} = (-4,1.5,2)$.}
	\label{ex:1}
\end{myfigure}

The first example shows that instances can have different performance depending on the choice of insertion rule.  Consider the simplex $P$ shown in Figure \ref{ex:1} (we present the coordinates of vertices in the figure's caption). We list the steps of Wolfe's method on $P$ for the minnorm and linopt insertion rules in Tables \ref{ex:1minnormsteps} and \ref{ex:1linoptsteps} and demonstrate a single step from each set of iterations in Figure \ref{ex:1steps}.  Each row lists major cycle and minor cycle iteration number, the vertices in the potential corral, and the value of $\ve{x}$ and $\ve{y}$ at the end of the iteration (before $\ve{x} = \ve{y}$ for major cycles).  Note that the vertex $\ve{p_4}$ is added to the potential corral twice with the linopt insertion rule, as evidenced in Table \ref{ex:1linoptsteps}.

\begin{table}
\centering
\begin{small}
\begin{tabular}{c c c c c}
	\toprule
	Major Cycle & Minor Cycle & $C$ & $\ve{x}$ & $\ve{y}$ \\ \midrule \midrule
	0 & 0 & $\{\ve{p_1}\}$ & $\ve{p_1}$ & \\ \midrule
	1 & 0 & $\{\ve{p_1},\ve{p_2}\}$ & $\ve{p_1}$ & $(1, 0.5,0)$ \\ \midrule
	2 & 0 & $\{\ve{p_1},\ve{p_2},\ve{p_3}\}$ & $(1,0.5,0)$ & $(0.3980,0.199,0.5473)$ \\ \midrule
	3 & 0 & $\{\ve{p_1},\ve{p_2},\ve{p_3},\ve{p_4}\}$ & $(0.3980,0.199,0.5473)$ & $(0,0,0)$ \\ \midrule
	3 & 1 & $\{\ve{p_1},\ve{p_2},\ve{p_4}\}$ & $(0.2878,0.1439,0.3957)$ & $(0.1980,0.0990,0.4455)$ \\ \bottomrule
\end{tabular}
\caption{iterations for \emph{minnorm} insertion rule}
\label{ex:1minnormsteps}
\end{small}
\end{table}

\begin{table}
\centering
\begin{small}
	\begin{tabular}{c c c c c}
		\toprule
		Major Cycle & Minor Cycle & $C$ & $\ve{x}$ & $\ve{y}$ \\ \midrule \midrule
		0 & 0 & $\{\ve{p_1}\}$ & $\ve{p_1}$ & \\ \midrule
		1 & 0 & $\{\ve{p_1},\ve{p_4}\}$ & $\ve{p_1}$ & $(0.2219,0.9723,0.2409)$ \\ \midrule
		2 & 0 & $\{\ve{p_1}, \ve{p_4},\ve{p_3}\}$ & $(0.2219,0.9723,0.2409)$ & $(0.2848,0.3417,0.5810)$ \\ \midrule
		2 & 1 & $\{\ve{p_1}, \ve{p_3}\}$ & $(0.2835, 0.3548,0.5739)$ & $(0.2774,0.3484,0.5807)$ \\ \midrule
		3 & 0 & $\{\ve{p_1},\ve{p_3},\ve{p_2}\}$ & $(0.2774,0.3484,0.5807)$ & $(0.3980,0.199,0.5473)$ \\ \midrule
		4 & 0 & $\{\ve{p_1},\ve{p_2},\ve{p_3},\ve{p_4}\}$ & $(0.3980,0.199,0.5473)$ & $(0,0,0)$ \\ \midrule
		4 & 1 & $\{\ve{p_1},\ve{p_2},\ve{p_4}\}$ & $(0.2878,0.1439,0.3957)$ & $(0.1980,0.0990,0.4455)$ \\ \bottomrule
	\end{tabular}
	\caption{iterations for \emph{linopt} insertion rule}
	\label{ex:1linoptsteps}
\end{small}
\end{table}

\begin{figure}[ht]
 \includegraphics[width=3in]{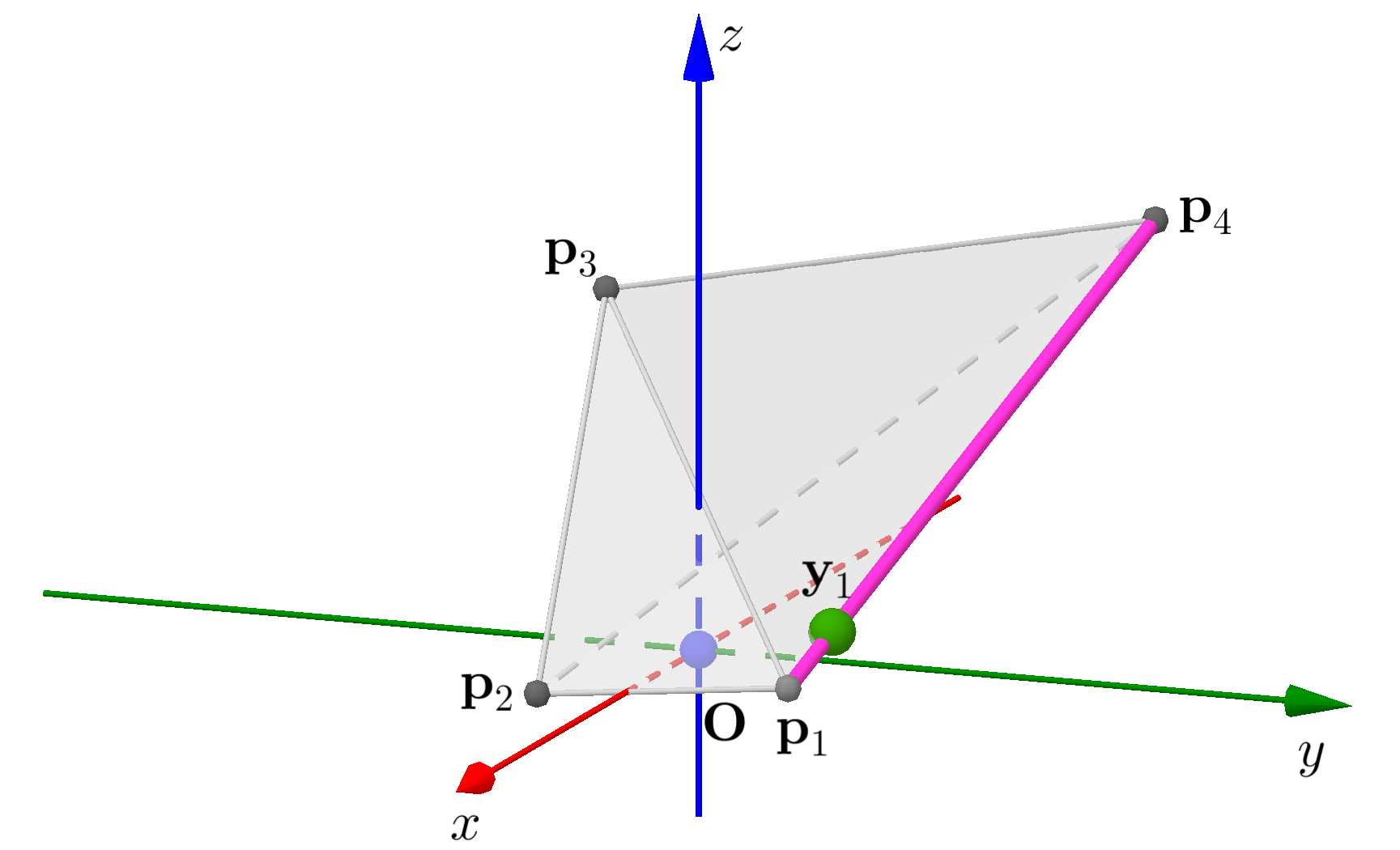} \includegraphics[width=3in]{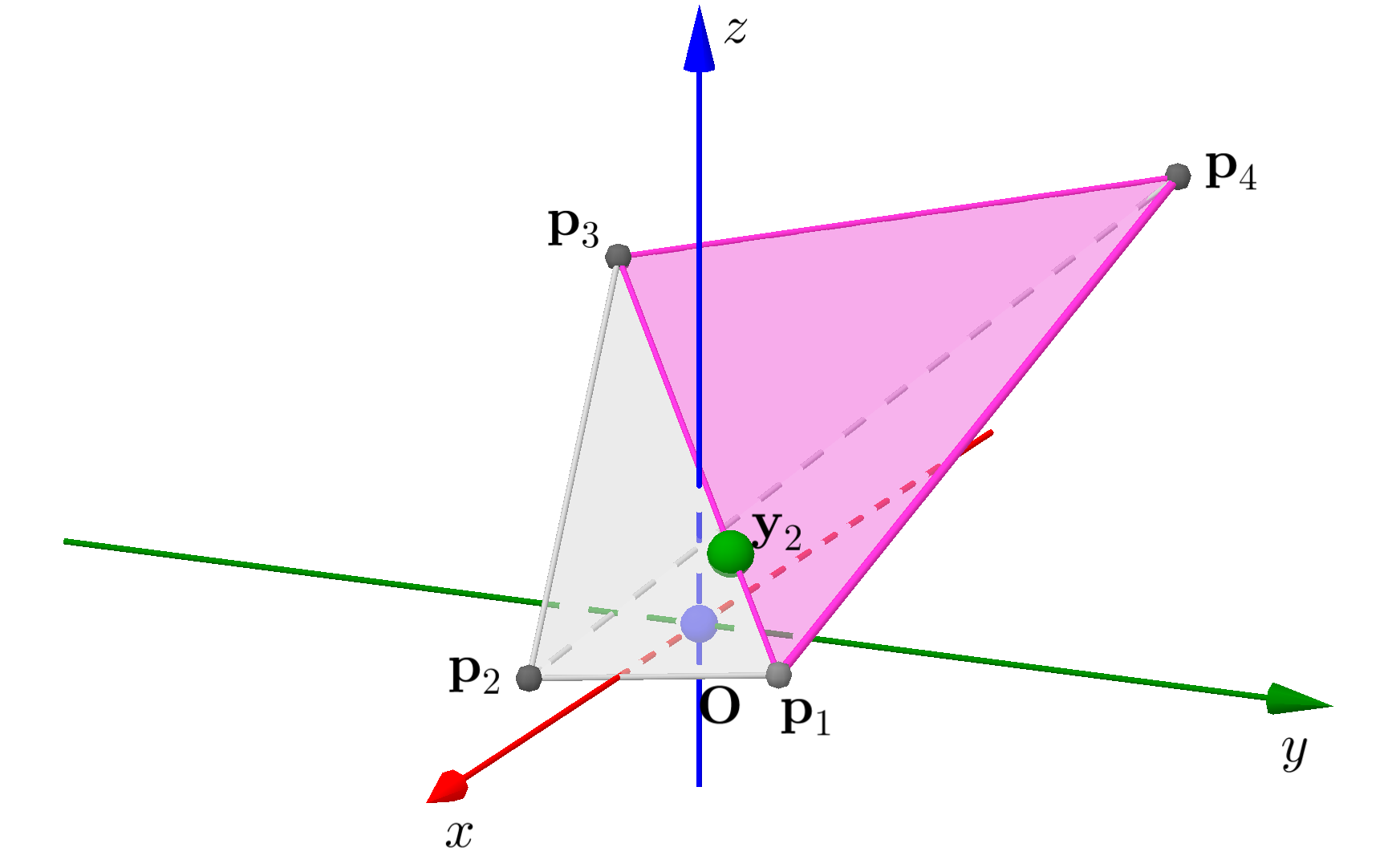}
 \caption{Left: Major cycle 1, minor cycle 0 for the linopt rule on $P$ illustrates the end of a major cycle; the affine minimizer $\ve{y_1} \in \text{relint}(\conv\{C\}) = \text{relint}(\conv\{\ve{p_1},\ve{p_4}\})$.  Right: Major cycle 2, minor cycle 0 for the linopt rule on $P$ illustrates the beginning of a minor cycle; the affine minimizer $\ve{y_2} \not\in \text{relint}(\conv\{C\}) = \text{relint}(\conv\{\ve{p_1},\ve{p_4},\ve{p_3}\})$ and the vertex $\ve{p_4}$ will be removed in the next minor cycle.}
 \label{ex:1steps} 
\end{figure}

\fi

Currently, there are examples of exponential behavior for the simplex method for all known deterministic pivot rules. 
It is our aim to provide the same for insertion rules on Wolfe's method.
In the next subsection we will present the first  exponential-time example using the minnorm insertion rule.

\subsection{An exponential lower bound for Wolfe's method}\label{sec:explowerbound}

To understand our hard instance, it is helpful to consider first a simple instance that shows an inefficiency of Wolfe's method. 
The example is a set of points where a point leaves and reenters the current corral: 
4 points in $\RR^3$, $(1,0,0), (1/2,1/4,1), (1/2,1/4,-1), (-2,1/4,0)$. 
If one labels the points $1,2,3,4$, the sequence of corrals with the minnorm rule is $1,12,23,234,14$, where point $1$ enters, leaves and reenters (For succintness, sets of points like $\{a,b,c\}$ may be denoted $abc$.). 
The idea now is to recursively replace point 1 (that reenters) in this construction by a recursively constructed set of points whose corrals are then considered twice by Wolfe's method. 
To simplify the proof, our construction uses a variation of this set of 4 points with an additional point and modified coordinates. This modified construction is depicted in \cref{fig:pq}, where point 1 corresponds to a set of points $\Pset{d-2}$, points 2,3 correspond to points $\ppoint{d}, \qpoint{d}$ and point 4 corresponds to points $\rpoint{d}, \spoint{d}$.

The high-level idea of our exponential lower bound example is the following.
We will inductively define a sequence of instances of increasing dimension of the minimum norm point problem. 
Given an instance in dimension $d-2$, we will add a few dimensions and points so that, when given to Wolfe's method, the number of corrals of the new augmented instance in dimension $d$ has about twice the number of corrals of the 
input instance in dimension $d-2$. More precisely, our augmentation procedure takes an instance $\Pset{d-2}$ in $\RR^{d-2}$, adds two new coordinates and adds four points, $\ppoint{d}, \qpoint{d}, \rpoint{d}, \spoint{d}$, to get an instance $\Pset{d}$ in $\RR^d$. 

Points $\ppoint{d}, \qpoint{d}$ are defined so that the method on instance $\Pset{d}$ goes first through every corral given by the points in the prior configuration $\Pset{d-2}$ and then goes to corral $\ppoint{d} \qpoint{d}$. 
To achieve this under the minimum norm rule, the four new points have greater norm than any point in $\Pset{d-2}$ and they are in the geometric configuration sketched in \cref{fig:pq}.

\begin{figure}[ht]
\centering
\includegraphics[width=.44\columnwidth]{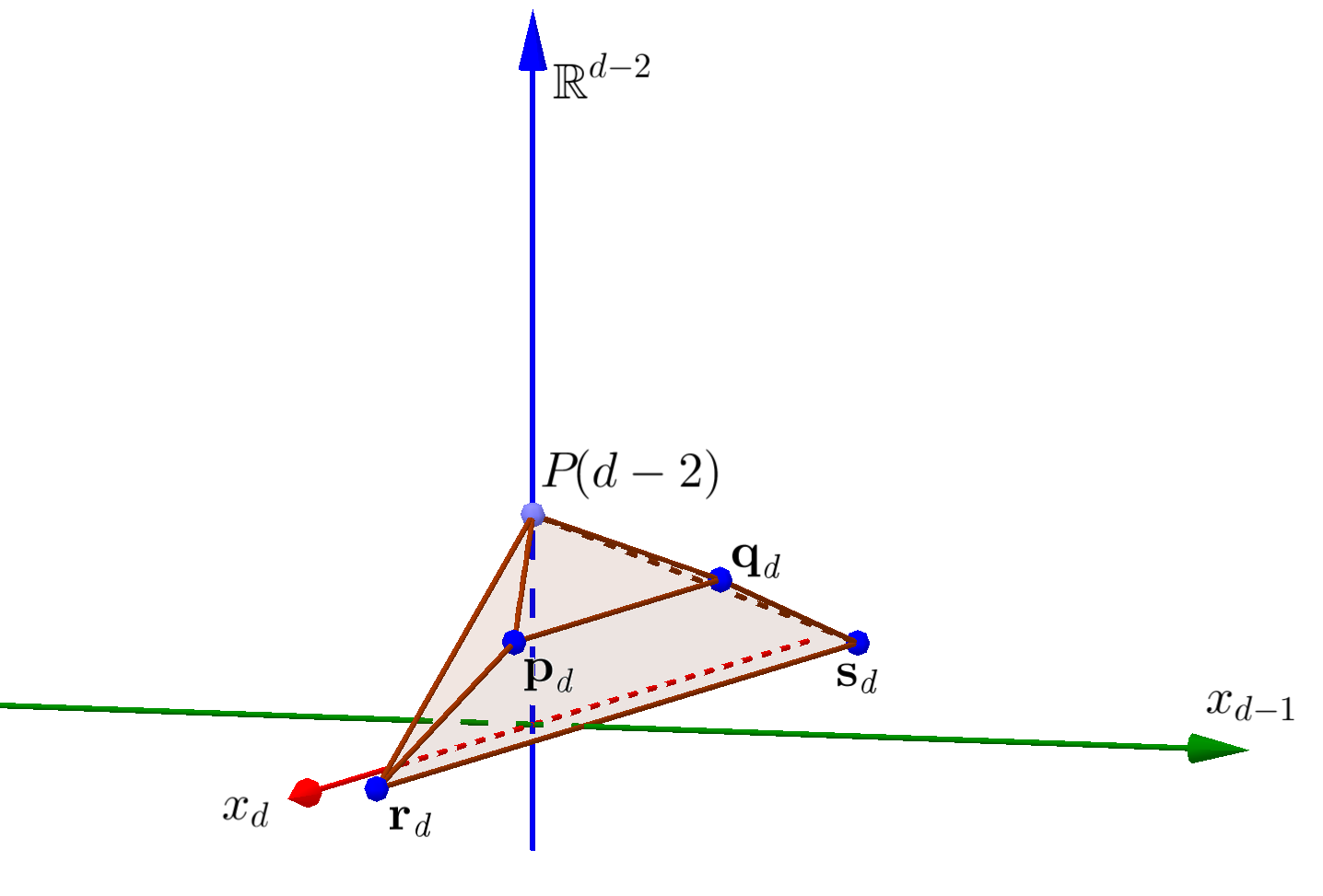}
\includegraphics[width=.46\columnwidth, trim=0 1in 0 0, clip]{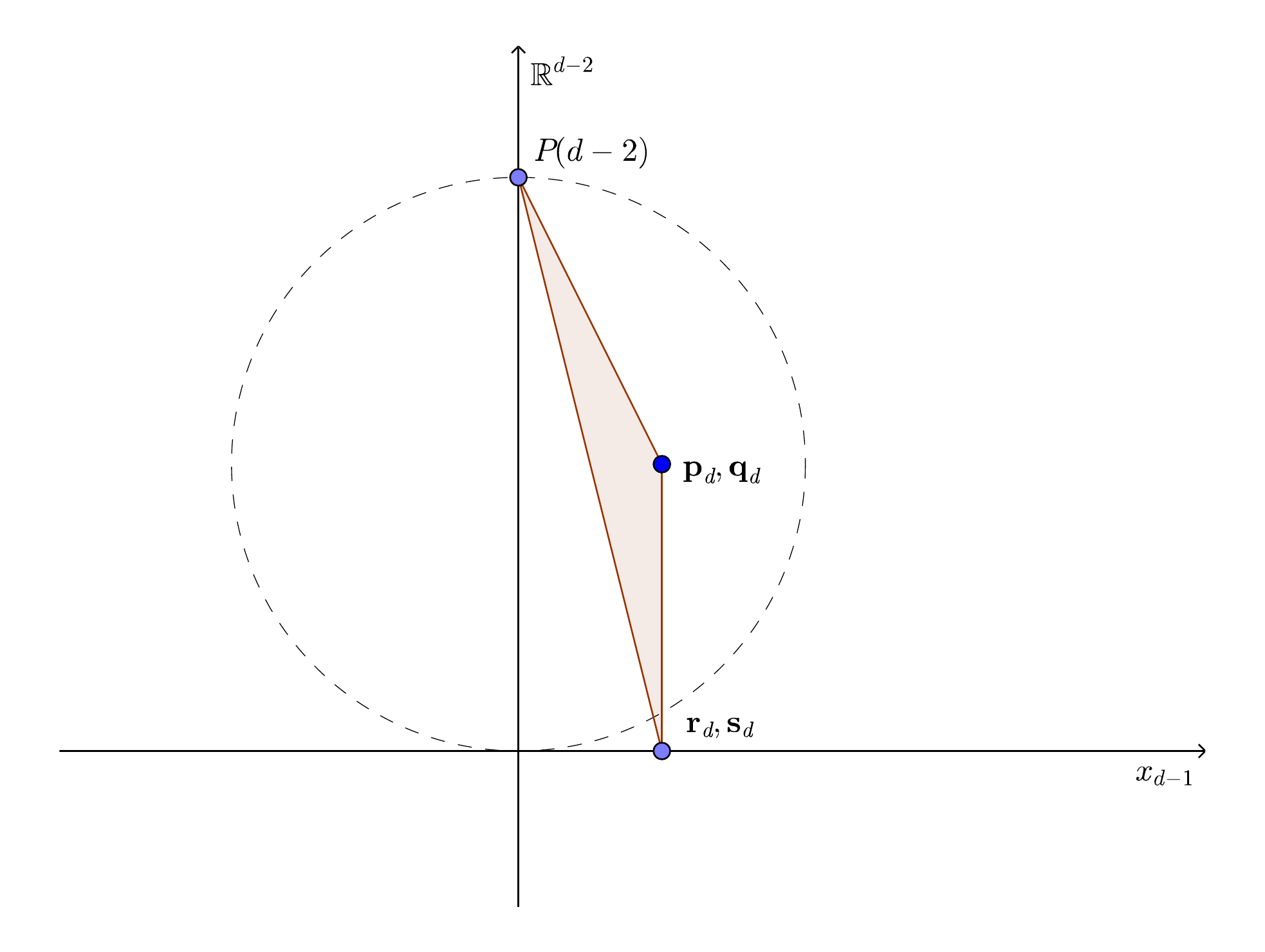}
\caption{Left: In this view of $\Pset{d}$, the point labeled $\Pset{d-2}$ represents all points from $\Pset{d-2}$ embedded into $\RR^d$.  The axis labeled $\RR^{d-2}$ represents the $(d-2)$-dimensional subspace, $\linspan{\Pset{d-2}}$ projected into $\linspan{\xstar{d-2}}$.  Right: A two-dimensional view of $\Pset{d}$ projected along the $x_d$ coordinate axis.}\label{fig:pq}
\end{figure}

At this time, no point in $\Pset{d-2}$ is in the current corral and so, if a point in $\Pset{d-2}$ is part of the optimal corral, 
it will have to reenter, which is expensive. Points $\rpoint{d}, \spoint{d}$ are defined so that $\rpoint{d} \spoint{d}$ is a corral after 
$\ppoint{d} \qpoint{d}$, but now every point in $\Pset{d-2}$ is improving according to Wolfe's criterion and may enter again.
Specifically, every corral in $\Pset{d-2}$, with $\rpoint{d} \spoint{d}$ appended, is visited again.

Before we start describing the exponential example in detail, we wish to review preliminary lemmas of independent interest which will be used in the arguments. The first lemma demonstrates that orthogonality between finite point sets allows us to easily describe the affine minimizer of their union.  Figure \ref{fig:affinemins} shows two such situations, one in which the affine hull of the union of the point sets span all of $\RR^3$ and one in which it does not.

\begin{figure}[ht]
\centering
\includegraphics[width=2.6in]{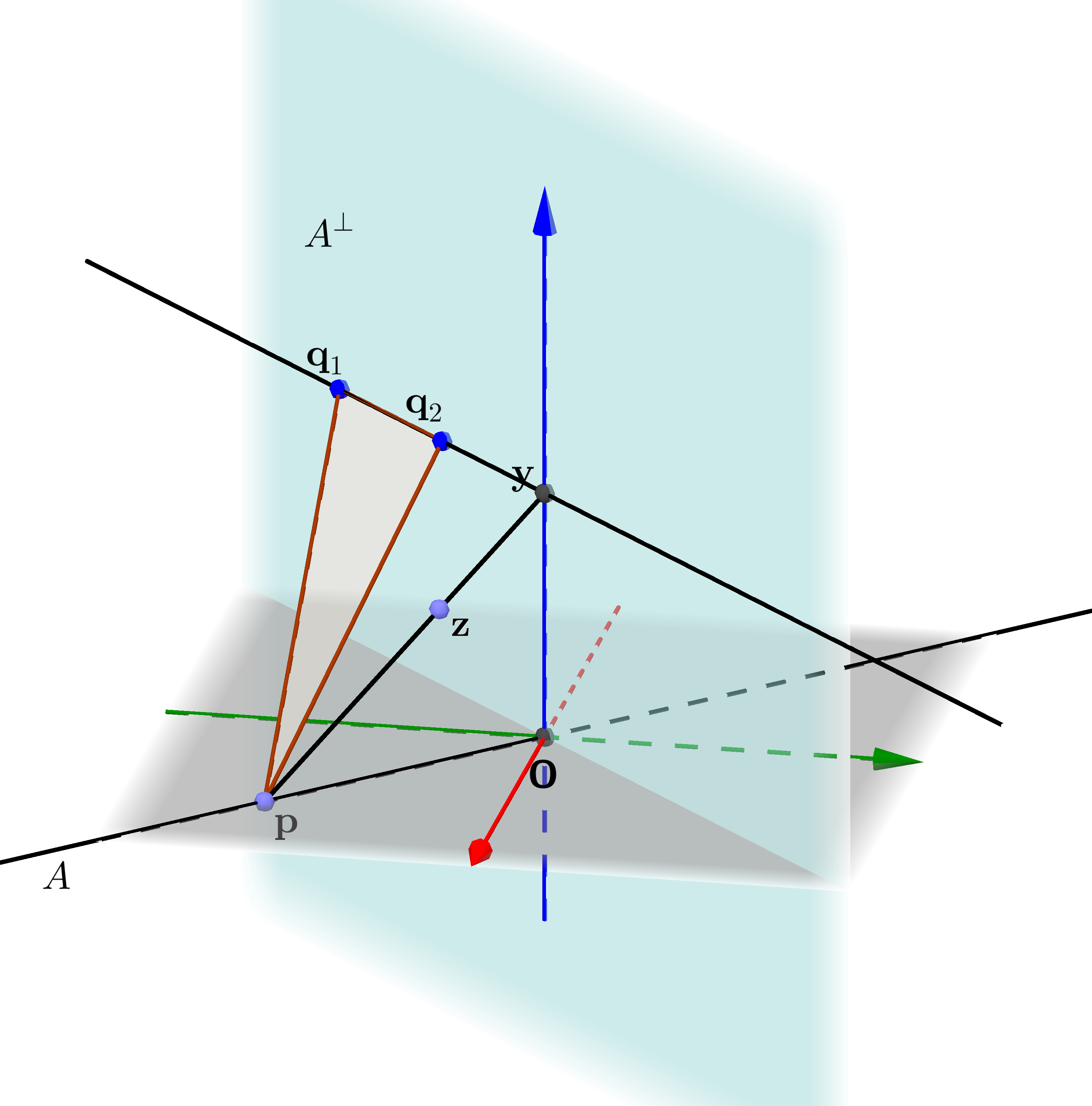}\includegraphics[width=2.6in]{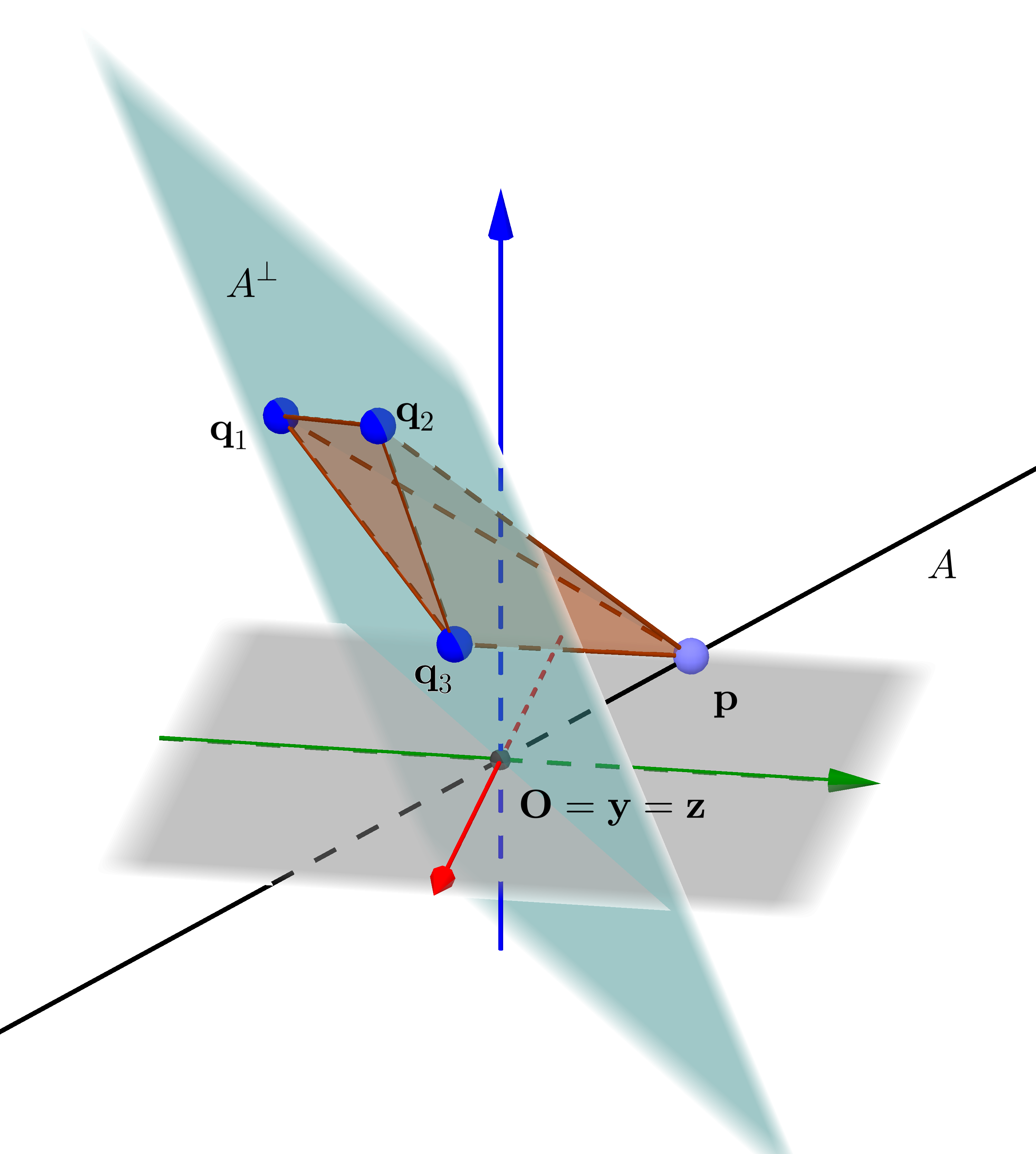}
\caption{Examples of Lemma \ref{lem:orthogonal}.  Left: the affine hull of $P\cup Q$ is not full dimensional, and thus the affine minimizer lies at $\ve{z}$ along the line segment connecting $\ve{x}=\ve{p}$ and $\ve{y}$.  Right: the convex hull of $P\cup Q$ is full-dimensional and thus the affine hull of $P\cup Q$ includes $O$ which is the affine minimizer.}
\label{fig:affinemins}
\end{figure}


\begin{lemma}\label{lem:orthogonal}
Let $A\subseteq \RR^d$ be a proper linear subspace. 
Let $P \subseteq A$ be a non-empty finite set. 
Let $Q \subseteq A^\perp$ be another non-empty finite set. 
Let $\ve{x}$ be the minimum norm point in $\aff P$. Let $\ve{y}$ be the minimum norm point in $\aff Q$. 
Let $\ve{z}$ be the minimum norm point in $\aff (P \cup Q)$.
We have:
\begin{enumerate}
\item 
$\ve{z}$ is the minimum norm point in $[\ve{x}, \ve{y}]$ and therefore, if $\ve x \neq \ve 0$ or $\ve y \neq \ve 0$, then $\ve{z} = \lambda \ve{x} + (1-\lambda) \ve{y}$ with $\lambda = \frac{\enorms{\ve{y}}}{\enorms{\ve{x}} + \enorms{\ve{y}}}$.

\item
If $\ve{x} \neq \ve{0}$ and $\ve{y} \neq \ve{0}$, then $\ve{z}$ is a strict convex combination of $\ve{x}$ and $\ve{y}$.

\item
If $\ve{x} \neq \ve{0}$, $\ve{y} \neq \ve{0}$ and $P$ and $Q$ are corrals, then $P \cup Q$ is also a corral.
\end{enumerate}
\end{lemma}
\begin{proof}
If $\ve x=\ve y=\ve 0$ then part 1 follows immediately. 
If at least one of $\ve x, \ve y$ is non-zero, then they are also distinct by the orthogonality assumption.
Given two distinct points $\ve{a}, \ve{b}$, one can show that the minimum norm point in the line through them is $\lambda \ve{a} + (1-\lambda)\ve{b}$ where $\lambda = \ve{b}^T (\ve{b}-\ve{a})/\enorms{\ve{b}-\ve{a}}$.
For points $\ve{x}, \ve{y}$ as in the statement, the minimum norm point in $\aff (\ve{x} \cup \ve{y})$ is $\ve{z}' = \lambda \ve{x} + (1-\lambda)\ve{y}$ with $\lambda = \frac{\enorms{\ve{y}}}{\enorms{\ve{x}} + \enorms{\ve{y}}} \in [0,1]$. Thus, $\ve{z}'$ is also the minimum norm point in $[\ve{x},\ve{y}]$.
We will now use the optimality condition in \cref{lem:affineoptimality} to conclude that $\ve{z}' = \ve{z}$.
Let $\ve{p} \in P$.
Then $\ve{p}^T \ve{z}'$ can be computed in two steps: 
First project $\ve{p}$ onto $\linspan{\ve{x},\ve{y}}$ (a subspace that contains $\ve{z}'$). This projection is $\ve{x}$ by optimality of $\ve{x}$. 
Then project onto $\ve{z}'$. This shows that $\ve{p}^T \ve{z}' = \ve{x}^T \ve{z}' = \enorms{\ve{z}'}$.
A similar calculation shows $\ve{q}^T \ve{z}' = \enorms{\ve{z}'}$ for any $\ve{q} \in Q$. 
We conclude that $\ve{z}'$ is the minimum norm point in $\aff (P \cup Q)$.
This proves part 1.

Part 2 follows from our expression for $\lambda$ above, which is in $(0,1)$ when $\ve{x} \neq \ve{0}$ and $\ve{y} \neq \ve{0}$.

Under the assumptions of part 3, we have that $\ve{x}$ is a strict convex combination of $P$ and $\ve{y}$ is a strict convex combination of $Q$. This combined with the conclusion of part 2 gives that $\ve{z}$ is a strict convex combination of $P\cup Q$. The claim in part 3 follows.
\end{proof}

The following lemma shows conditions under which, if we have a corral and a new point that only has components along the minimum norm point of the corral and along new coordinates, then the corral with the new point added is also a corral. Moreover, the new minimum norm point is a convex combination of the old minimum norm point and the added point. Figure \ref{fig:pluspointcorral} gives an example of such a situation in $\RR^3$. Denote by $\linspan{M}$ the linear span of the set $M$.

\begin{myfigure}
\centering
\includegraphics[width=3.1in]{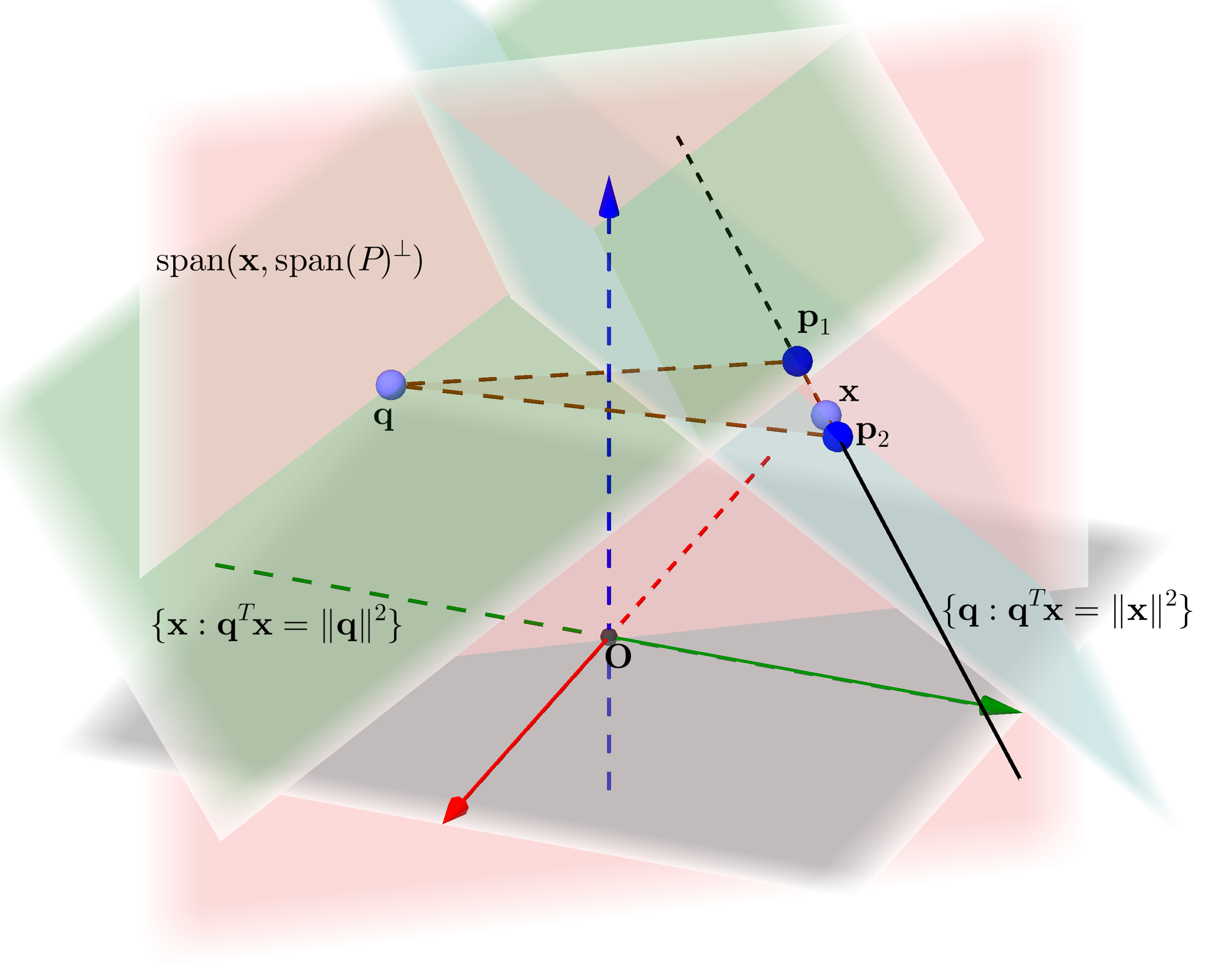}
\caption{An example of Lemma \ref{lem:corralpluspoint} in which point $\ve{q}$ satisfies all assumptions and $P \cup \{\ve{q}\}$ is a corral.  The hyperplanes are labeled with their defining properties and demonstrate that $\ve{q}^T\ve{x} < \min\{\|\ve{x}\|^2,\|\ve{q}\|^2\}.$ The minimizer of $P \cup \{\ve{q}\}$ lies at the intersection of the blue, vertical axis and $\conv(P \cup \{\ve{q}\}).$}
\label{fig:pluspointcorral}
\end{myfigure}

\begin{lemma}\label{lem:corralpluspoint}
Let $P \subseteq \RR^d$ be a finite set of points that is a corral. Let $\ve{x}$ be the minimum norm point in $\aff P$.
Let $\ve{q} \in \operatorname{span} \bigl(\ve{x}, \linspan{P}^\perp\bigr)$, and assume $\ve{q}^T\ve{x} < \min \bigl\{\enorms{\ve{q}}, \enorms{\ve{x}} \bigr\}$. Then $P \cup \{\ve{q}\}$ is a corral. Moreover, the minimum norm point $\ve y$ in  $\conv (P \cup \{\ve{q}\})$ is a (strict) convex combination of $\ve{q}$ and the minimum norm point of $P$: $\ve{y} = \lambda \ve{x} + (1-\lambda)\ve{q}$ with $\lambda = \ve{q}^T (\ve{q}-\ve{x})/\enorms{\ve{q}-\ve{x}}$.
\end{lemma}
\details{Condition is essentially necessary, though the case $\ve{x} = \ve{q}$ needs a bit of care.}
\begin{proof}
Let $\ve{y}$ be the minimum norm point in $\aff ({P \cup \{\ve{q}\}})$.
Intuitively, $\ve{y}$ should be the minimum norm point in the line through $\ve{x}$ and $\ve{q}$.
We will characterize $\ve{y}$ and show that it is a strict convex combination of $P \cup \{\ve{q}\}$ (which implies that it is a corral).
Given two points $\ve{a}, \ve{b}$, one can show that the minimum norm point in the line through them is $\lambda \ve{a} + (1-\lambda)\ve{b}$ where $\lambda = \ve{b}^T (\ve{b}-\ve{a})/\enorms{\ve{b}-\ve{a}}$.
Thus, we define $\ve{y} = \lambda \ve{x} + (1-\lambda)\ve{q}$ with $\lambda = \ve{q}^T (\ve{q}-\ve{x})/\enorms{\ve{q}-\ve{x}}$.
By definition we have $\ve{y} \in \aff(P \cup \{\ve{q}\})$.

The minimality of the norm of $\ve{y}$ follows from the optimality condition in Lemma \ref{lem:affineoptimality}. 
It holds by construction for $\ve{q}$. It also holds for $\ve{p} \in P$: 
The projection of $\ve{p}$ onto $\ve{y}$ can be computed in two steps. 
First, project onto $\linspan{\ve{x},\ve{q}}$ (a subspace that contains $\ve{y}$), which is $\ve{x}$ by optimality of $\ve{x}$. 
Then project onto $\ve{y}$. 
This shows that $\ve{p}^T \ve{y} = \ve{x}^T \ve{y} = \norms{\ve{y}}$ (the second equality by optimality of $\ve{y}$). 
We conclude that $\ve{y}$ is of minimum norm in $\aff{P \cup \{\ve{q}\}}$.

To conclude that $P \cup \{\ve{q}\}$ is a corral, we show that $\ve{y}$ is a strict convex combination of points $P \cup \{\ve{q}\}$.
It is enough to show that $\ve{y}$ is a strict convex combination of $\ve{x}$ and $\ve{q}$.
We have $\lambda = \ve{q}^T (\ve{q}-\ve{x})/\enorms{\ve{q}-\ve{x}} = \frac{\enorms{\ve{q}} - \ve{q}^T\ve{x}}{\enorms{\ve{q}-\ve{x}}} > 0$ by assumption.
We also have $1-\lambda = -\ve{x}^T (\ve{q}-\ve{x})/\enorms{\ve{q}-\ve{x}} = \frac{\enorms{\ve{x}} - \ve{q}^T\ve{x}}{\enorms{\ve{q}-\ve{x}}} > 0$ by assumption.
\end{proof}

Our last lemma shows that if we have points in two orthogonal subspaces, $A$ and $A^\perp$, then adding a point from $A^\perp$ to a set from $A$ does not cause any points from $A$ that previously did not violate Wolfe's criterion (for the affine minimizer) to violate it.  Figure \ref{fig:halfspaceintersection} demonstrates this situation.

\begin{myfigure}
\centering
\includegraphics[width=3in]{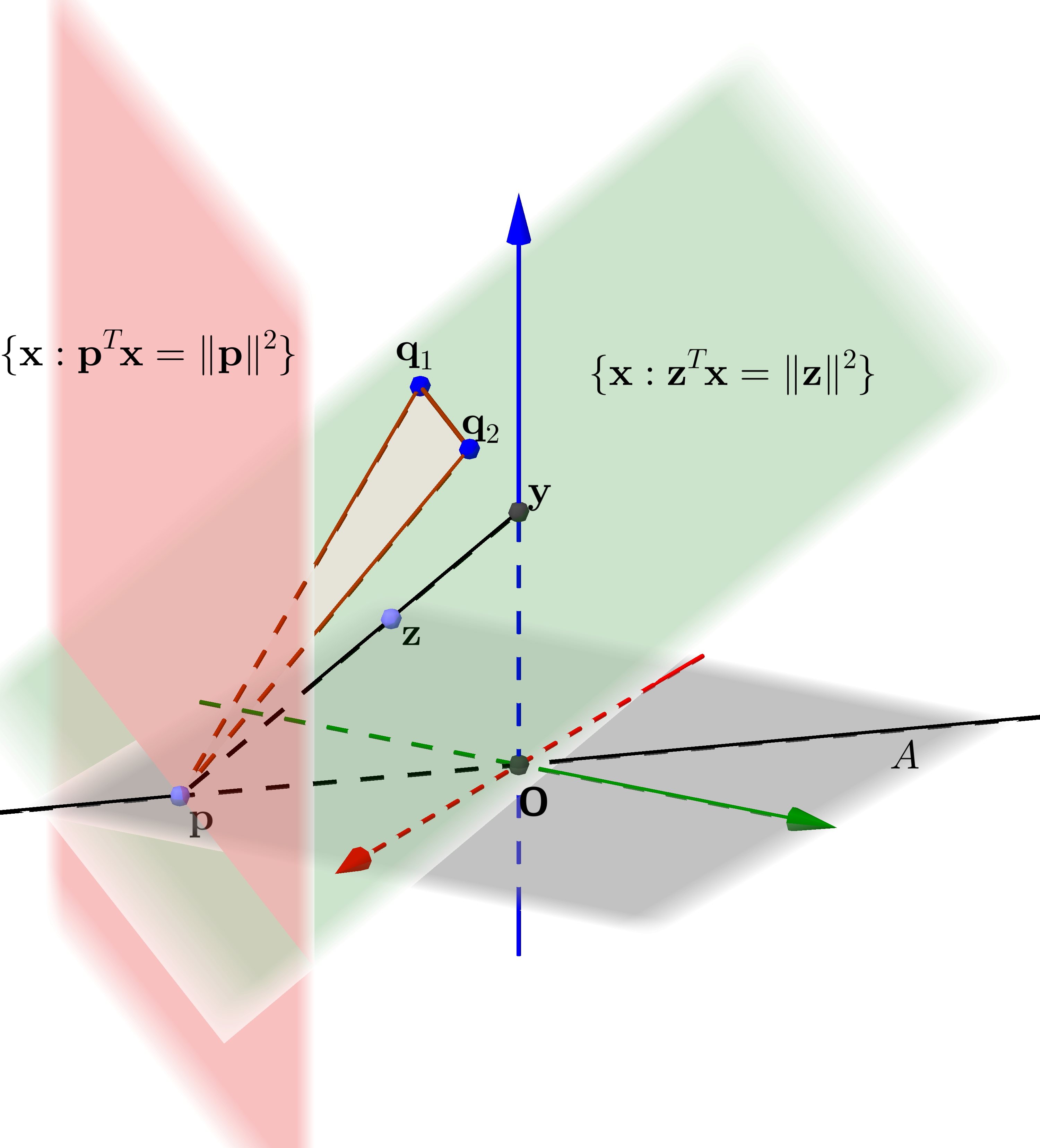}
\caption{An example of Lemma \ref{lem:secondround} in which adding points $Q$ from $A^\perp$ to points $P$ from $A$ create a new affine minimizer, $\ve{z}$, but the points satisfying Wolfe's criterion in $A$ remain the same.  Note that both hyperplanes intersect at the affine minimizer of $P$, so the halfspace intersections with $A$ are the same.}
\label{fig:halfspaceintersection}
\end{myfigure}

\begin{lemma}\label{lem:secondround}
For a point $\ve{z}$ define $H_{\ve{z}}=\{\ve{w} \in \RR^n \suchthat \ve{w}^T \ve{z} < \enorms{\ve{z}} \}$.
Suppose that we have an instance of the minimum norm point problem in $\RR^d$ as follows: 
Some points, $P$, live in a proper linear subspace $A$ and some, $Q$, in $A^\perp$. 
Let $\ve{x}$ be the minimum norm point in $\aff P$ and $\ve{y}$ be the minimum norm point in $\aff(P \cup Q)$. Then $H_{\ve{y}} \cap A = H_{\ve{x}} \cap A$.
\end{lemma}
\begin{proof}
Let $B$ be the span of $\ve{x}$ and $Q$. We first show $\ve{y} \in B$. 
To see this, suppose not. 
Decompose $\ve{y}$ as $\ve{y}=\lambda \ve{v} + \sum_{\ve{q} \in Q} \mu_q \ve{q}$, where $\ve{v} \in \aff P$ and $\lambda + \sum \mu_q = 1$.
Decompose $\ve{v}$ as $\ve{v} = \ve{u} + \ve{x}$ where $\ve{u} \perp \ve{x}$ and $\ve{u} \in A$ (this is possible because $\ve{v}-\ve{x}$ is orthogonal to $\ve{x}$, by optimality of $\ve{x}$, \cref{lem:affineoptimality}).
Thus, $\ve{y} = \lambda \ve{u} + \lambda \ve{x} + \sum_{\ve{q} \in Q} \mu_q \ve{q}$ with $\lambda \ve{u}$ orthogonal to $\lambda \ve{x} + \sum_{\ve{q} \in Q} \mu_q \ve{q}$.
This implies that $\ve{y}' = \lambda \ve{x} + \sum_{\ve{q} \in Q} \mu_q \ve{q}$ has a smaller norm than $\ve{y}$ and $\ve{y}' \in \aff (P \cup Q)$. 
This is a contradiction.

To conclude, we have $H_{\ve{y}} \cap A$ is a halfspace in $A$ whose normal is parallel to the projection of $\ve{y}$ onto $A$
(It is helpful to understand how to compute the intersection of a hyperplane with a subspace. 
If $T_{\ve{g}} = \{\ve{w} : \ve{w} \cdot \ve{g} = 1\}$ and $S$ is a linear subspace, then $T_{\ve{g}} \cap S = \{\ve{w} \in S : \ve{w} \cdot \proj_S \ve{g} = 1 \}$. 
In other words, in order to intersect a hyperplane with a subspace we project the normal.) 
That is, it is parallel to $\ve{x}$. 
But that halfspace must also contain $\ve{x}$ on its boundary.
Thus, that halfspace is equal to $H_{\ve{x}} \cap A$.
\end{proof}

We will now describe our example in detail. The simplest version of our construction uses square roots and real numbers. We present instead a version with a few additional tweaks so that it only involves rational numbers.

Let $\Pset{1} = \{ 1\} \subseteq \QQ$. For odd $d >1$, let $\Pset{d}$ be a list of points in $\QQ^d$ defined inductively as follows: 
Let $\xstar{d}$ denote the minimum norm point in $\conv \Pset{d}$.  
Let $\M{d} := \max_{\ve{p} \in \Pset{d}} \norm{\ve{p}}_1$, which is a rational upper bound to the maximum $2$-norm among the points in $\Pset{d}$. 
(For a first reading one can let $\M{d}$ be the maximum $2$-norm among points in $\Pset{d}$, which leads to an essentially equivalent instance except that it is not rational.) 
Similarly, let $\m{d} = \norm{\xstar{d}}_\infty$, which is a rational lower bound to the minimum norm among points in $\conv \Pset{d}$. 
(Again, for a first reading one can define $\m{d} = \enorm{\xstar{d}}$ which leads to an essentially equivalent 
instance, except that it is not rational.) 

We finally present the example. If we identify $\Pset{d}$ with a matrix where the points are rows, then the points in $\Pset{d}$ are given by the following block matrix:

\[
\Pset{d} =
\begin{pmatrix}
\Pset{d-2} & 0 & 0 \\
\frac{1}{2} \xstar{d-2} & \frac{\m{d-2}}{4} & \M{d-2} \\
\frac{1}{2} \xstar{d-2} & \frac{\m{d-2}}{4} & -(\M{d-2}+1) \\
0 & \frac{\m{d-2}}{4} & \M{d-2}+2 \\
0 & \frac{\m{d-2}}{4} & -(\M{d-2}+3).
\end{pmatrix}.
\]

The last four rows of the matrix $\Pset{d}$ are the points $\ppoint{d}, \qpoint{d}, \rpoint{d}, \spoint{d}$ of the configuration. For a picture of the case of $\Pset{3}$ see Figure \ref{fig:pqdim3}.
\begin{figure}[ht]
\includegraphics[width=3.4in]{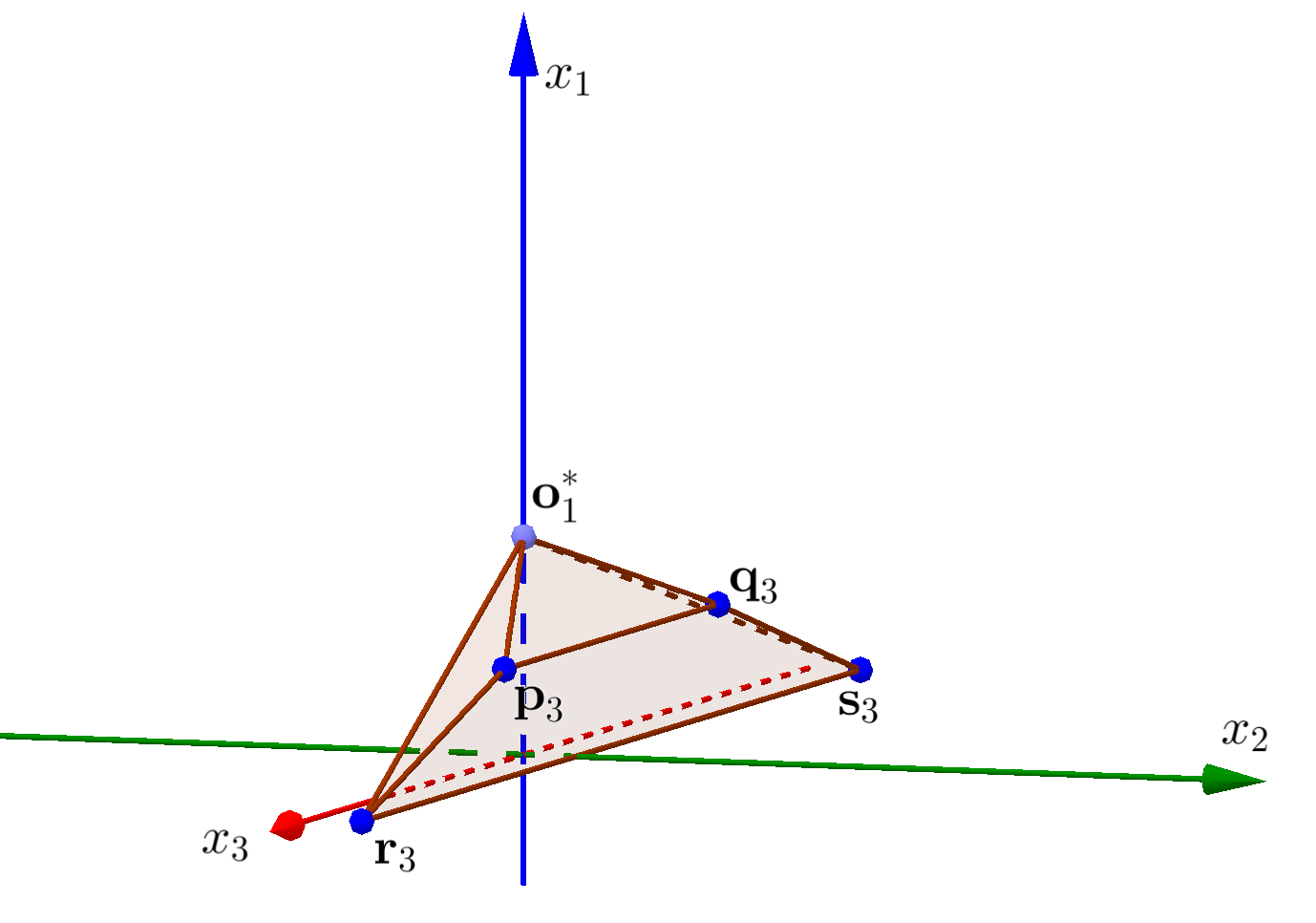}
\includegraphics[width=3in]{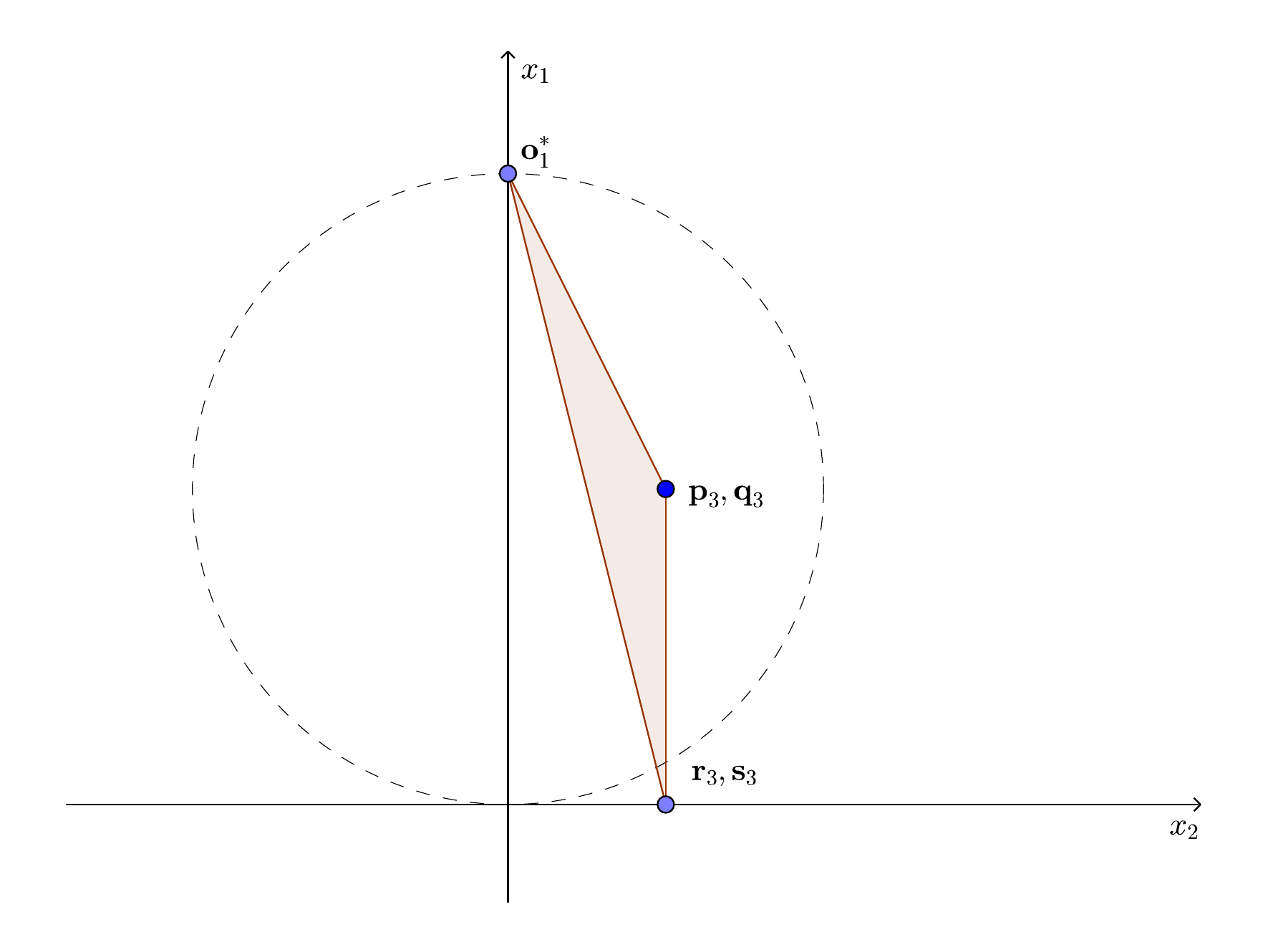}
\caption{Left: Three-dimensional view of $\Pset{3}$. Right: A two-dimensional view of $\Pset{3}$ projected along the $x_3$ coordinate axis.}\label{fig:pqdim3}
\end{figure}

\begin{remark}
First note that strictly speaking $\Pset{d-2} \subset \QQ^{d-2}$, and that we are defining an embedding of it into
$\QQ^d$, for which we have to use a recursive process. To avoid unnecessary notation, we will abuse the notation. The point $\ve{v_{d-2}}$ denotes both a point of $\Pset{d-2}$ and of the subsequent $\Pset{d}$, i.e., $\ve{v_{d-2}}=(\ve{v},0,0)$ will be the 
identical copy of $\ve{v_{d-2}}$ within $\Pset{d}$, but we add two extra zero coordinates. Depending on the context 
$\ve{v_{d-2}}$ will be understood as both a $(d-2)$-dimensional vector or as a $d$-dimensional vector (e.g., when
doing dot products). The points of $\Pset{d-2}$ become a subset of the point configuration $\Pset{d}$ by padding extra zeros.  See Figures \ref{fig:pq} and \ref{fig:expexcloud} which illustrate this embedding and address our visualizations of these sets in three dimensions.
\end{remark}

\begin{myfigure}[ht]
\centering
\includegraphics[width=4.0in]{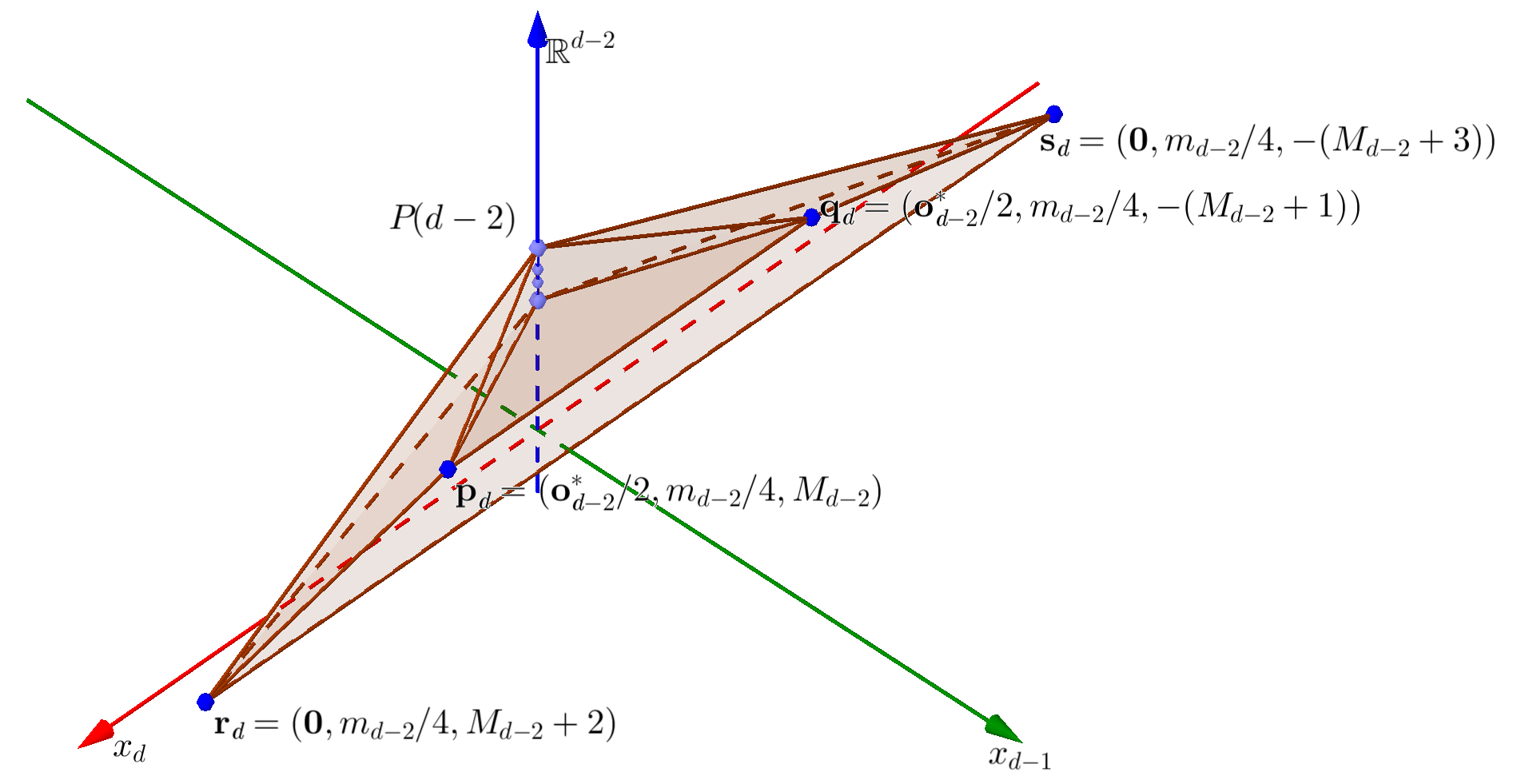}
\caption{As described in Figure \ref{fig:pq}, the axis labeled $\RR^{d-2}$ represents the $(d-2)$-dimensional subspace $\linspan{\Pset{d-2}}$ projected onto the one dimensional subspace $\linspan{\xstar{d-2}}$.  Here we illustrate that the projection of the set $\Pset{d-2}$ forms a `cloud' of points and the convex hull of this projection has many fewer faces than the unprojected convex hull.  For simplicity, we will visualize $\Pset{d-2}$ and subsets of $\Pset{d-2}$ as a single point in $\linspan{\xstar{d-2}}$ as in Figure \ref{fig:pq}.}
\label{fig:expexcloud}
\end{myfigure}

\newcommand{\thmlength}{5 \cdot 2^{k-1} -4}
\begin{theorem}\label{thm:lowerbound}
Consider the execution of Wolfe's method with the \emph{minnorm} point rule on input $\Pset{d}$ where $d=2k-1$.
Then the sequence of corrals has length $\thmlength$.
\end{theorem}
\begin{proof}
Points in $\Pset{d}$ are shown in order of increasing norm.
Let $\ppoint{d}, \qpoint{d}, \rpoint{d}, \spoint{d}$ denote the last four points of $\Pset{d}$, respectively. 
Let $\Cset{d}$ denote the ordered sequence of corrals in the execution of Wolfe's method on $\Pset{d}$. 
Let $\Oset{d}$ denote the last (optimal) corral in $\Cset{d}$.


The rest of the proof will establish that the sequence of corrals $\Cset{d}$ is
\begin{align*}
\Cset{d-2} \\
\Oset{d-2} \ppoint{d} \\
\ppoint{d} \qpoint{d} \\
\qpoint{d} \rpoint{d} \\
\rpoint{d} \spoint{d} \\
\Cset{d-2} \rpoint{d} \spoint{d}
\end{align*}
(where a concatenation such as $\Cset{d-2} \rpoint{d} \spoint{d}$ denotes every corral in $\Cset{d-2}$ with $\rpoint{d}$ and $\spoint{d}$ added).
After this sequence of corrals is established, we solve the resulting recurrence relation: Let $T(d)$ denote the length of $\Cset{d}$. We have $T(1) = 1$, $T(d) = 2 T(d-2) + 4$. This implies $T(d) = \thmlength$ (with $d=2k-1$).


All we must show now to complete the proof of Theorem \ref{thm:lowerbound} is that $\Cset{d}$ has indeed
the stated recursive form. We do this by induction on $d$. The steps of the proof are written as claims with individual proofs.

By construction, $\Cset{d}$ starts with $\Cset{d-2}$. 
This happens because points in $\Cset{d}$ are ordered by increasing norm and the proof proceeds inductively 
as follows:
The first corral in $\Cset{d}$ is the minimum norm point in $\Pset{d}$, which is also the first corral in $\Cset{d-2}$.
Suppose now that the first $t$ corrals of $\Cset{d}$ coincide with the first $t$ corrals of $\Cset{d-2}$.
We will show that corral $t+1$ in $\Cset{d}$ is the same as corral $t+1$ in $\Cset{d-2}$. 
To see this, it is enough to see that the set of points in $\Pset{d}$ that can enter (improving points) contains 
the point that enters in $\Cset{d-2}$ (with two zeros appended) and contains no point of smaller norm.
This two-part claim is true because the two new zero coordinates play no role in this and all points in $\Pset{d}$ 
but not in $\Pset{d-2}$ have a larger norm than any point in $\Pset{d}$.

Once $\Oset{d-2}$ is reached (with minimum norm point $\xstar{d-2}$), the set of improving points, as established by Wolfe's criterion, consist of $\{\ppoint{d}, \qpoint{d}, \rpoint{d}, \spoint{d}\}$. Now, because we are using the minimum-norm insertion rule, the next point to enter is $\ppoint{d}$.

\begin{claim}\label{claim:star}
$\Oset{d-2} \ppoint{d}$ is a corral.
\end{claim}
\begin{claimproof}
This is a special case of \cref{lem:corralpluspoint}.
We have $\ppoint{d} = (\xstar{d-2}/2, \m{d-2}/4,\M{d-2})$.
We just need to verify the two inequalities in \cref{lem:corralpluspoint}:
\[
(\xstar{d-2})^T{\ppoint{d}} = \enorms{\xstar{d-2}}/2 < \enorms{\xstar{d-2}} < \enorms{\ppoint{d}}.
\]
\end{claimproof}
\begin{claim}
The next improving point to enter is $\qpoint{d}$.
\end{claim}

\begin{claimproof}
\details{Old version:
From \cref{lem:secondround}, taking $P= \Pset{d-2}$, $A$ equal to its linear span and $Q$ the points $p_d,q_d,r_d,s_d$, we can conclude that the next point which may be available to enter is $q_d$, since nothing in $\Pset{d-2}$ is available. The first thing to observe is that $\qpoint{d}$ is closer to the origin than $\rpoint{d},\spoint{d}$, so it is enough to 
check that $\qpoint{d}$ is an improving point per Wolfe's criterion.
From \cref{lem:corralpluspoint} we know the optimal point $\ve{y}$ in the corral $\Oset{d-2} \ppoint{d}$ 
explicitly  in terms of the optimal point $\xstar{d-2}$ of $\Oset{d-2}$ and $\ppoint{d}$, namely 
$\ve{y}$ is a convex combination $\lambda \xstar{d-2}+ (1-\lambda) \ppoint{d}$, with $\lambda=\frac{\enorms{\ppoint{d}} - {\ppoint{d}}^T\xstar{d-2}}{\enorms{\ppoint{d}-\xstar{d-2}}}$. Using Wolfe's criterion 
we can decide whether point $\qpoint{d}$ is a candidate to enter the corral. Compute



}
We first check that no point in $\Pset{d-2}$ can enter. 
From \cref{lem:corralpluspoint} we know the optimal point $\ve{y}$ in corral $\Oset{d-2} \ppoint{d}$ 
explicitly  in terms of the optimal point $\xstar{d-2}$ of $\Oset{d-2}$ and $\ppoint{d}$, namely 
$\ve{y}$ is a convex combination $\lambda \xstar{d-2}+ (1-\lambda) \ppoint{d}$, with $\lambda=\frac{\enorms{\ppoint{d}} - {\ppoint{d}}^T\xstar{d-2}}{\enorms{\ppoint{d}-\xstar{d-2}}}$. 
Let $\ve{p} \in \Pset{d-2}$. We check that it cannot enter via Wolfe's criterion. We compute $\ve{p}^T \ve{y}$ in two steps:
First project $\ve{p}$ onto $\linspan{ \xstar{d-2}, \ppoint{d}}$ (a subspace that contains $\ve{y}$).
This projection is longer than $\xstar{d-2}$ by optimality of $\xstar{d-2}$.
Then project onto $\ve{y}$.
This shows that $\ve{p}^T \ve{y} \geq {\xstar{d-2}}^T \ve{y} = \enorms{\ve{y}}$ and $\ve{p}$ cannot enter as it is not an improving point according to Wolfe's criterion.

By construction, $\qpoint{d}$ is closer to the origin than $\rpoint{d},\spoint{d}$, so to conclude it is enough to 
check that $\qpoint{d}$ is an improving point per Wolfe's criterion.
Compute




\begin{align*}
 \ve{y}^T\qpoint{d} 
 &= \lambda(\xstar{d-2})^T\qpoint{d} + (1-\lambda) \ppoint{d}^T\qpoint{d} \\
 &\leq \frac{\lambda}{2} \enorms{\xstar{d-2}} + (1-\lambda)\left[\frac{1}{4} \enorms{\xstar{d-2}} + \frac{1}{16}\enorms{\xstar{d-2}} -\M{d-2}^2-\M{d-2}\right] \\
 &\leq \frac{\lambda}{2} \enorms{\xstar{d-2}} 
\end{align*}
since by construction $\M{d-2} \geq 1$ and $\enorm{\xstar{d-2}} \leq 1$.
On the other hand,
\begin{align*}
\enorms{\ve{y}} 
&= \lambda^2 \enorms{\xstar{d-2}} + 
(1-\lambda)^2 \enorms{\ppoint{d}} +
2\lambda(1-\lambda)\frac{1}{2}\enorms{\xstar{d-2}} \\
&= \lambda \enorms{\xstar{d-2}} + 
(1-\lambda)^2 \enorms{\ppoint{d}} \\
&\geq \lambda \enorms{\xstar{d-2}}.
\end{align*}
Thus, $\ve{y}^T\qpoint{d} < \norms{\ve{y}}$, that is, $\qpoint{d}$ is an improving point.
%
\end{claimproof}

\begin{claim}
The current set of points, $\Oset{d-2} \cup \{\ppoint{d}, \qpoint{d}\}$, is not a corral. 
Points in $\Oset{d-2}$ leave one by one. The next corral is $\ppoint{d} \qpoint{d}$. 
\end{claim}
\begin{claimproof}
Instead of analyzing the iterations of Wolfe's inner loop, we use the key fact,
from \cref{Wolfeintro}, that the inner loop must end with a corral whose distance to 
the origin is strictly less than the previous corral. We look at the alternatives:
This new corral cannot be $\Oset{d-2} \cup \{\ppoint{d}\}$ (the previous corral)  or any 
subset of it because it would not decrease the distance.
An analysis similar to that of \cref{claim:star} or basic trigonometry (in three-dimensions) 
shows that $\Oset{d-2} \cup \{\qpoint{d}\}$  is a corral whose distance to the origin 
is larger than the distance for 
$\Oset{d-2} \cup \{\ppoint{d}\}$. See \cref{jamie1}, where we show a projection, the perpendicular line segments to $\conv(\Oset{d-2},\ppoint{d})$ and $\conv(\Oset{d-2},\qpoint{d})$ are shown in dotted line after projection.  
Thus, the new corral cannot be $\Oset{d-2} \cup \{\qpoint{d}\}$ or any subset of it.

No set of the form $S \cup \{\ppoint{d}, \qpoint{d}\}$ with $S \subseteq \Oset{d-2}$ and 
$S$ non-empty can be a corral: 
To see this, first note that the minimum norm point in $\conv(S \cup \{\ppoint{d}, \qpoint{d}\})$ 
is in the segment $[\ppoint{d}, \qpoint{d}]$, specifically, point $(\xstar{d-2}/2, \m{d-2}/4,0)$ 
(minimality follows from Wolfe's criterion, Lemma \ref{wolfec}).
This implies that the minimum norm point in $\aff (S \cup \{\ppoint{d}, \qpoint{d}\})$ cannot be in the relative interior of $\conv(S \cup \{\ppoint{d}, \qpoint{d}\})$ when $S$ is non-empty (see Figure \ref{jamie1star}).

The only remaining non-empty subset is $\{\ppoint{d}, \qpoint{d}\}$, which is the new corral.
\end{claimproof}

\begin{figure}
\begin{center}
\includegraphics[scale=0.3]{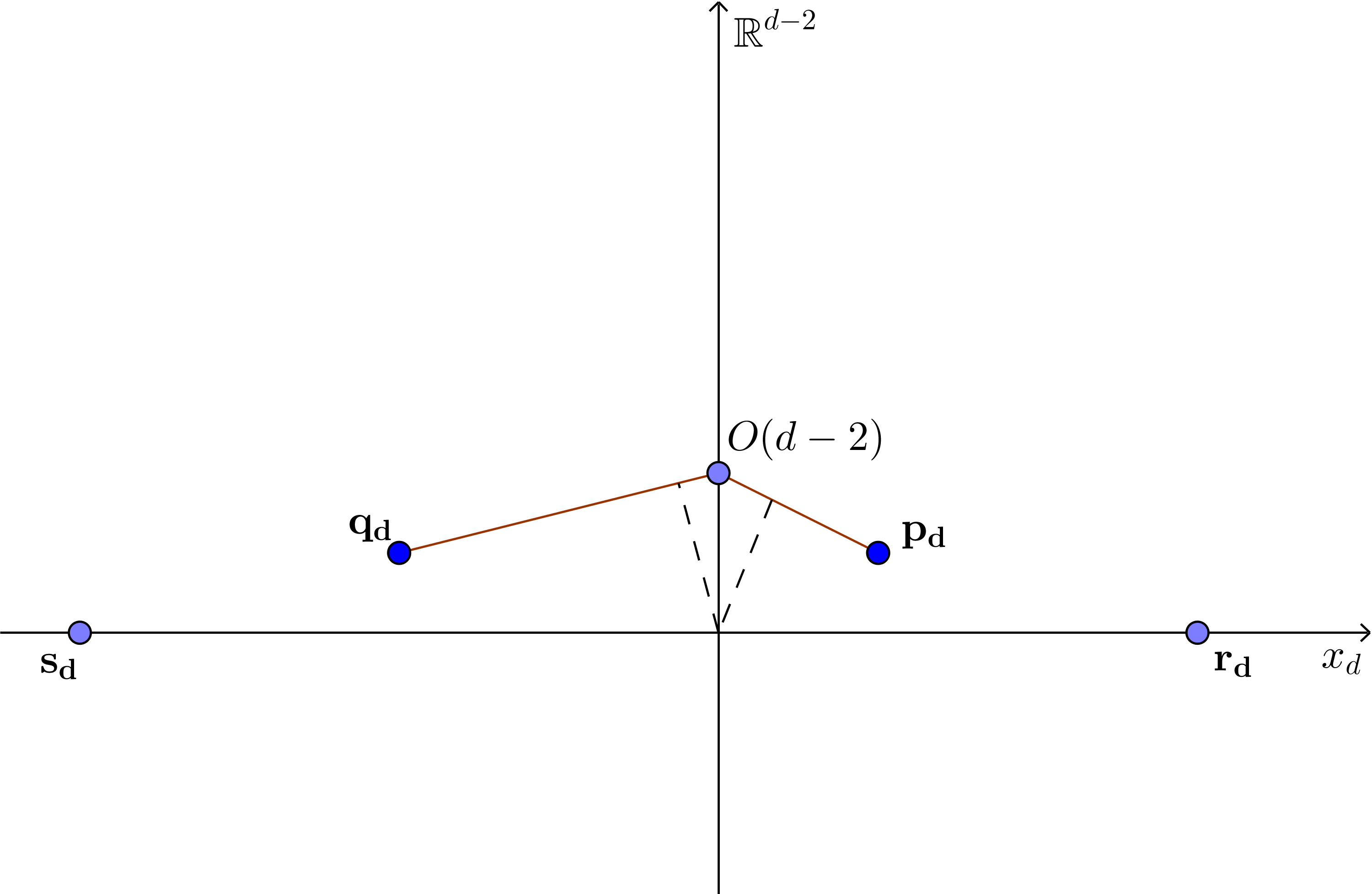}
\end{center}
\caption{A projection of the point set in the direction of $x_{d-1}$. Any corral of the form $S\qpoint{d}$ where $S \subset \Oset{d-2}$ would have distance larger than the previous corral, $\Oset{d-2}\ppoint{d}$.}\label{jamie1}
\end{figure}

\begin{figure}
\includegraphics[scale=0.73]{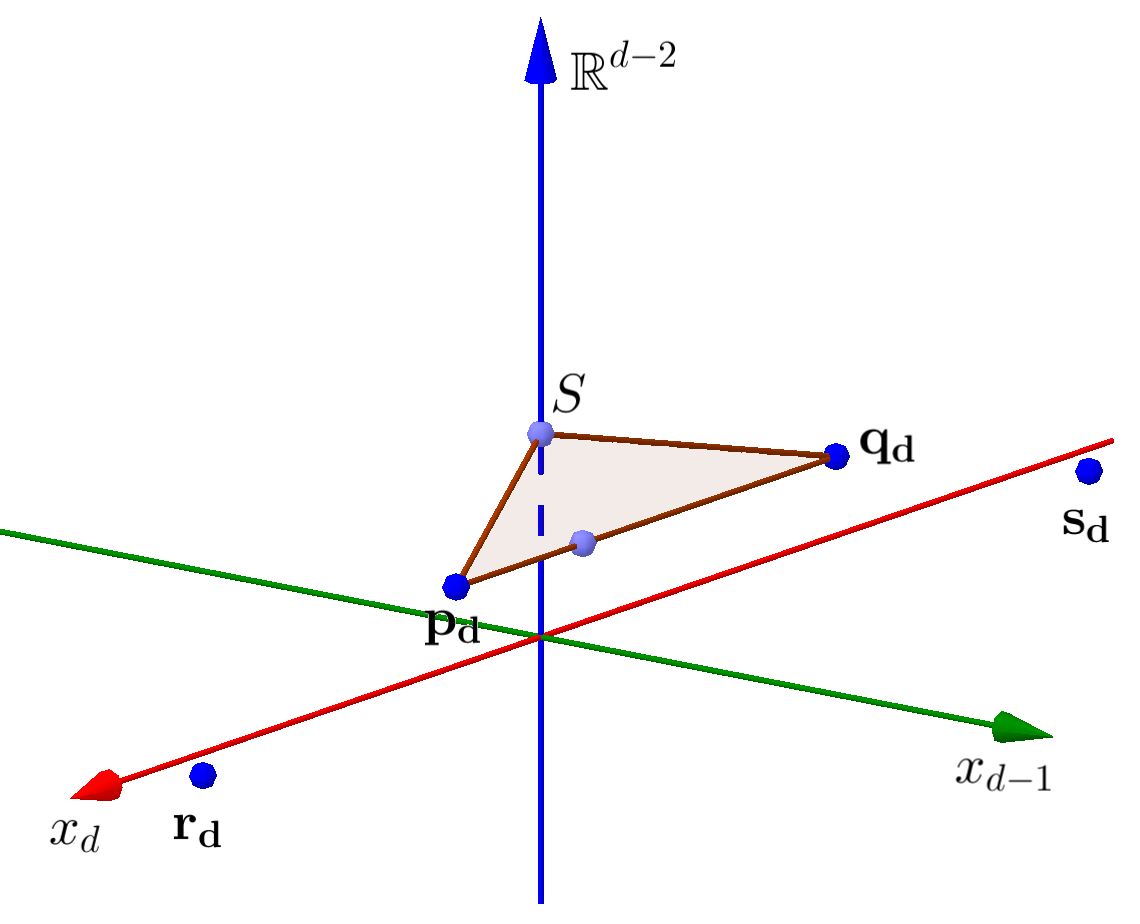}
\includegraphics[scale=0.29]{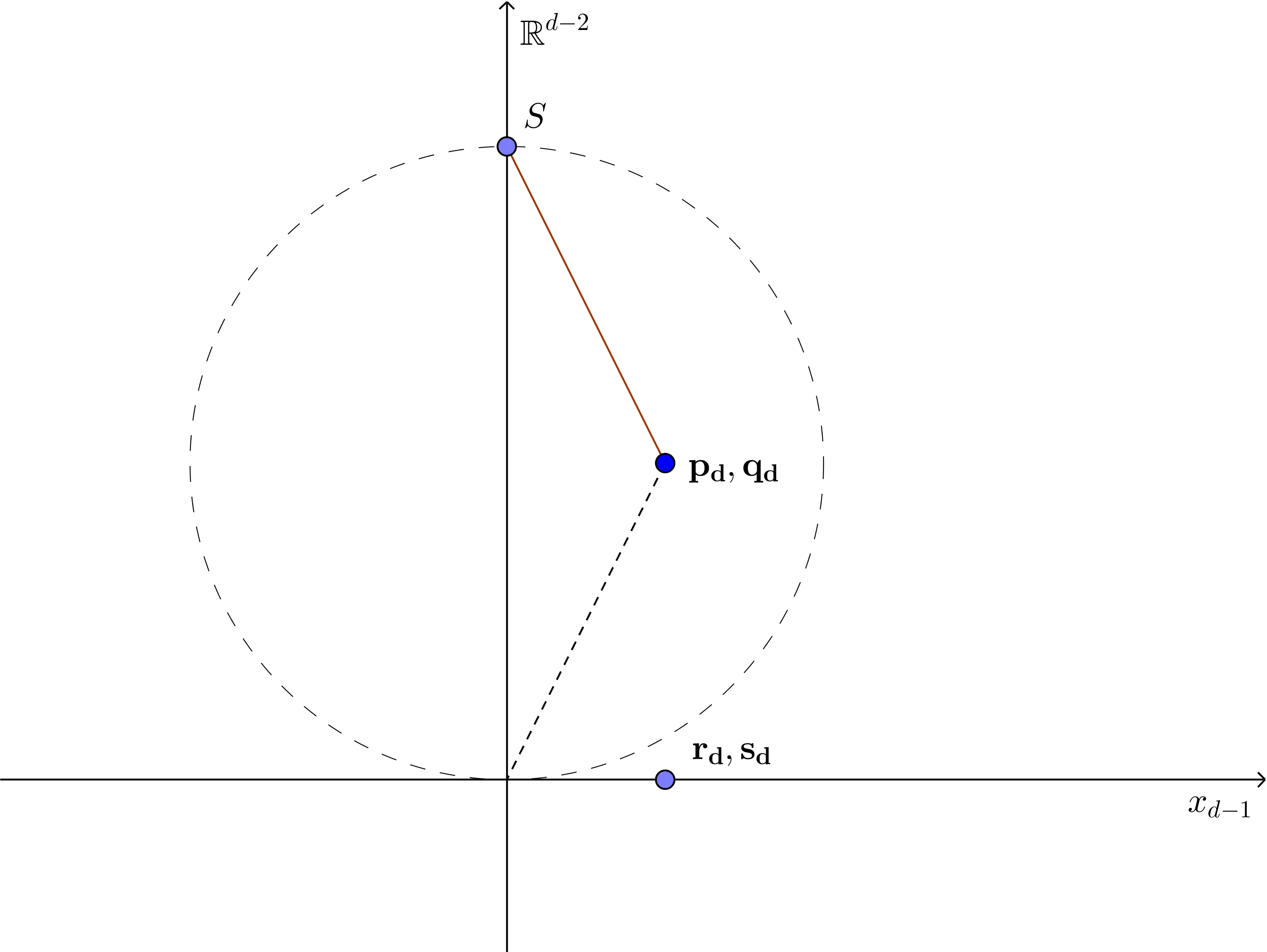}
\caption{The minimum norm point in $\conv(S \cup \{\ppoint{d}, \qpoint{d}\})$ is in the line segment between $\ppoint{d}$ and $\qpoint{d}$.}\label{jamie1star}
\end{figure}

\begin{claim}
The set of improving points is now $\{\rpoint{d}, \spoint{d}\}$. 
\end{claim}
\begin{claimproof}
Recall that the optimal point in corral $\{ \ppoint{d},\qpoint{d} \}$ has coordinates $(\xstar{d-2}/2, \m{d-2}/4,0)$. 
Thus, when computing distances and checking Wolfe's criterion it is enough to do so in the two-dimensional situation 
depicted in Figure \ref{jamie2}. Thus,
a hyperplane orthogonal to the segment from the origin to $(\xstar{d-2}/2, \m{d-2}/4,0)$ is shown in the figure. It leaves the points in $\Pset{d-2}$ above and both $\rpoint{d}$ and $\spoint{d}$ below making them the only improving points.
\end{claimproof}

\begin{figure}
\begin{center}
\includegraphics[scale=0.3]{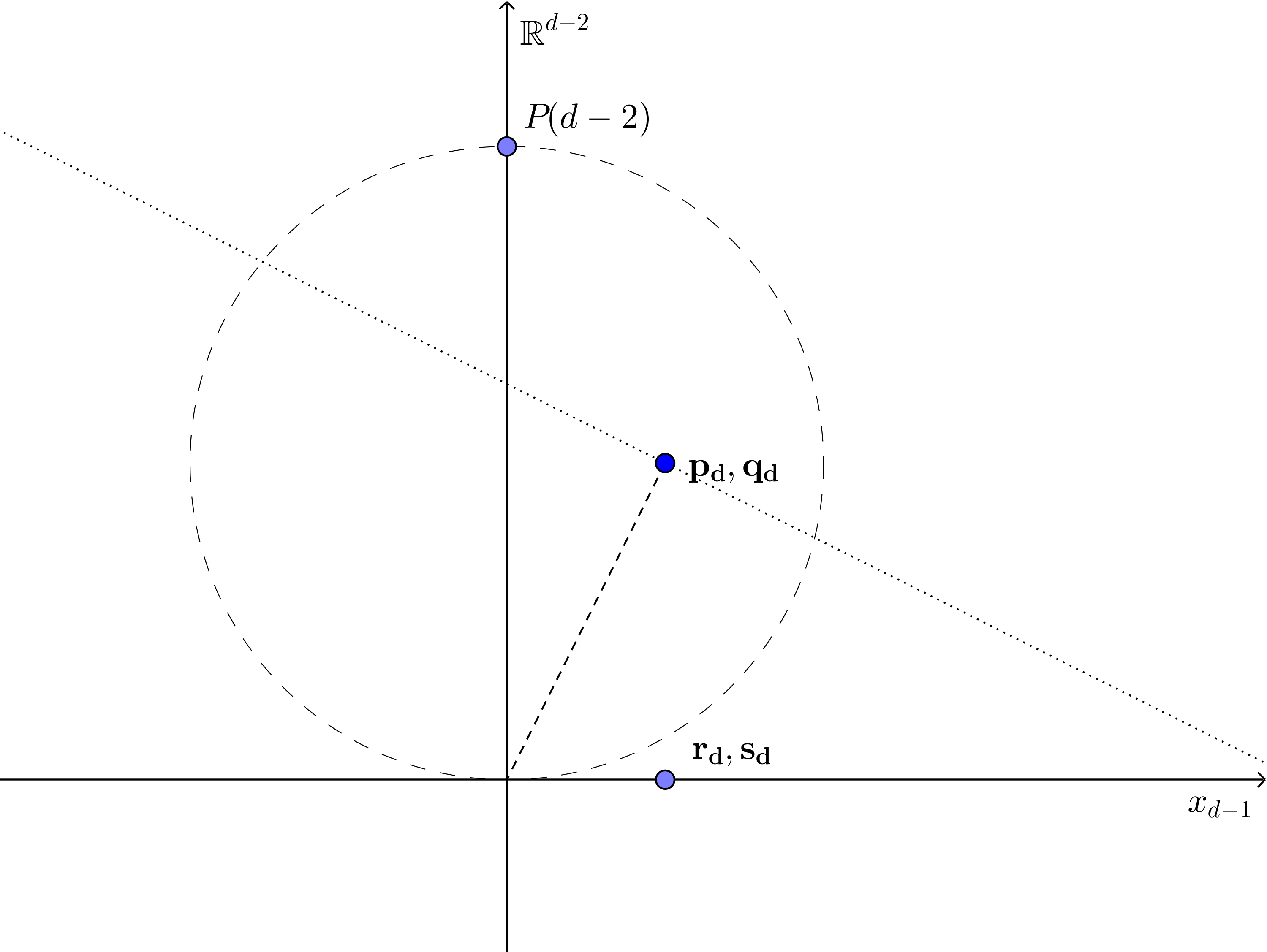}
\end{center}
\caption{The set of improving points is now $\{\rpoint{d},\spoint{d}\}$.}\label{jamie2}
\end{figure}

Point $\rpoint{d}$ enters since it has smallest norm.
\begin{claim}
Point $\ppoint{d}$ leaves and the next corral is $\qpoint{d} \rpoint{d}$. 
\end{claim}
\begin{claimproof}
To start, notice that by construction the four points $\ppoint{d},\qpoint{d},\rpoint{d},\spoint{d}$ lie on a common hyperplane, $L$, parallel to the subspace spanned by
$\xstar{d-2}$. Thus, one does not need to do distance calculations but rather  \cref{jamie1prime} is a faithful representation of the positions of points. The origin is facing the hyperplane $L$, parallel to $\linspan{\xstar{d-2}}$.  The closest point to the origin 
within $L$ is in the line segment joining $\rpoint{d},\spoint{d}$ thus, as we move vertically, the closest point to the origin in triangle $\ppoint{d},\qpoint{d},\rpoint{d}$ must be in the line segment joining 
$\rpoint{d}$ and $\qpoint{d}$.
\end{claimproof}

\begin{figure}
\begin{minipage}[t]{0.49\textwidth}
\begin{center}
\includegraphics[scale=0.28]{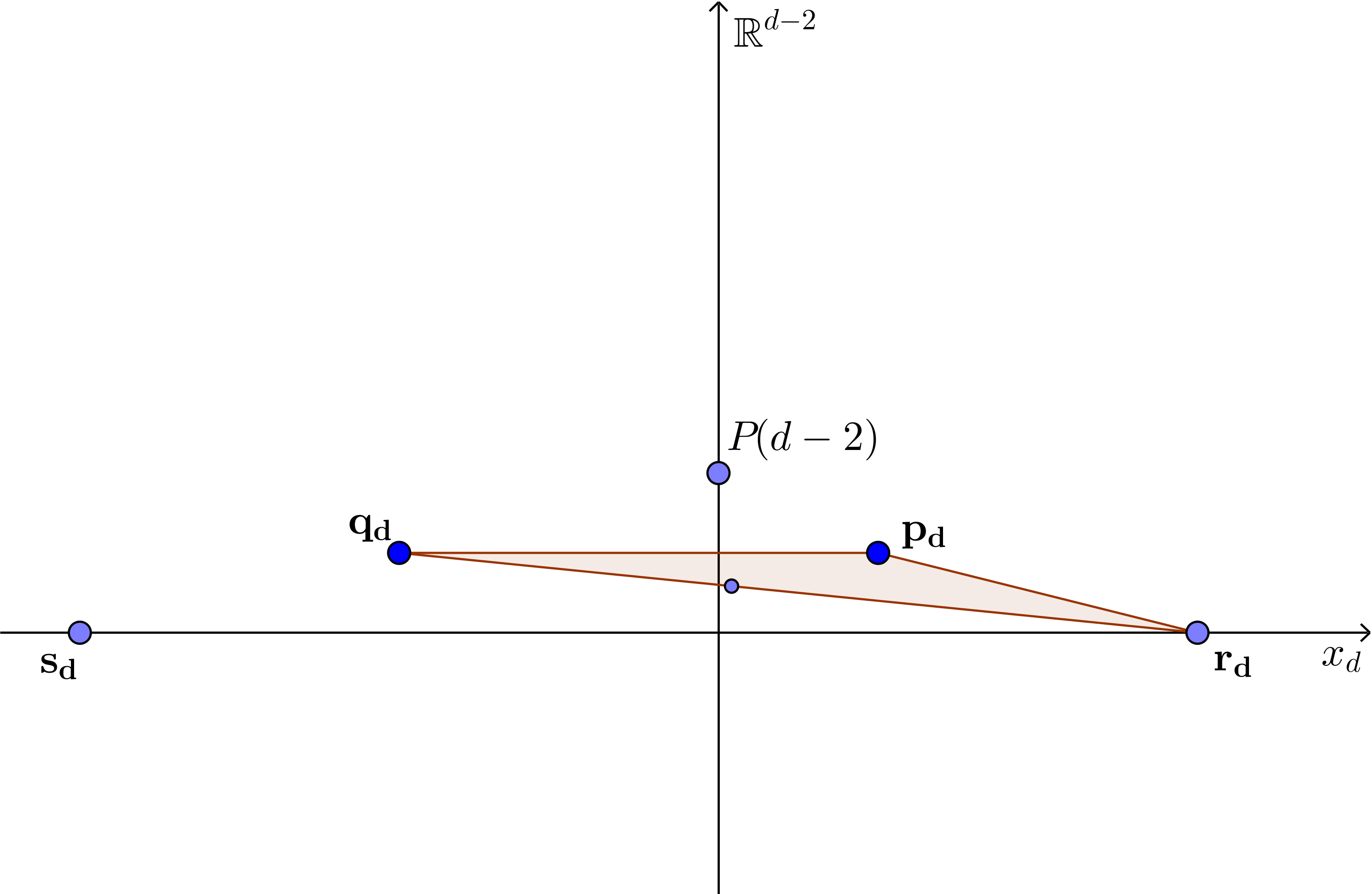}
\end{center}
\caption{The set $\{\ppoint{d},\qpoint{d},\rpoint{d}\}$ is not a corral.}\label{jamie1prime}
\end{minipage}
\begin{minipage}[t]{0.49\textwidth}
\begin{center}
\includegraphics[scale=0.25]{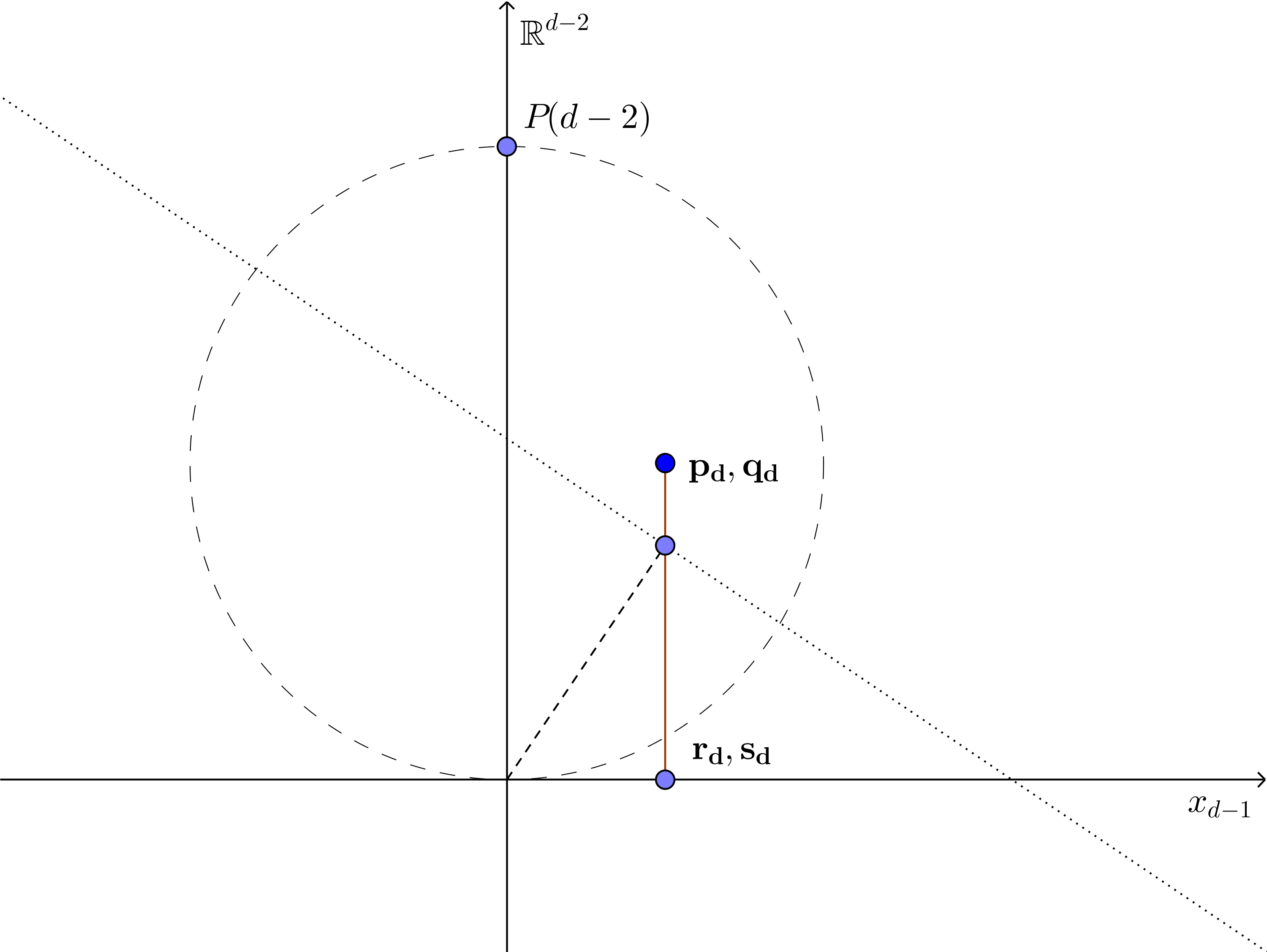}
\end{center}
\caption{The only improving point is $\spoint{d}$.}\label{jamietwoprime}
\end{minipage}
\end{figure}

\begin{claim}
The only improving point now is $\spoint{d}$. 
\end{claim}
\begin{claimproof}
Once more we rely in two different orthogonal two-dimensional projections of $\Pset{d}$ to
estimates distances and to check Wolfe's criterion. The line segment  from the origin to the optimum of the corral $\qpoint{d},\rpoint{d}$ (we
could calculate this exactly, but it is not necessary), and
its orthogonal hyperplane are shown in Figure \ref{jamietwoprime}.  This shows only $\rpoint{d}$ or $\spoint{d}$ are
candidates but $\rpoint{d}$ is in the current corral, so only $\spoint{d}$ may be added.
\end{claimproof}

Point $\spoint{d}$ enters as the closest improving point to the origin.
\begin{claim}
Point $\qpoint{d}$ leaves.
The next corral is $\rpoint{d} \spoint{d}$. 
\end{claim}
\begin{claimproof}
We wish to find the closest point to the origin in triangle $\qpoint{d},\rpoint{d},\spoint{d}$. From
Figure \ref{jamiethree} the optimum is between $\rpoint{d},\spoint{d}$; this point is $(\m{d-2}/4) \ve{e_{d-1}}$. Clearly
this point is below $\qpoint{d}$, so it must leave the corral. 
\end{claimproof}

\begin{figure}
\begin{minipage}{0.49\textwidth}
\begin{center}
\includegraphics[scale=0.28]{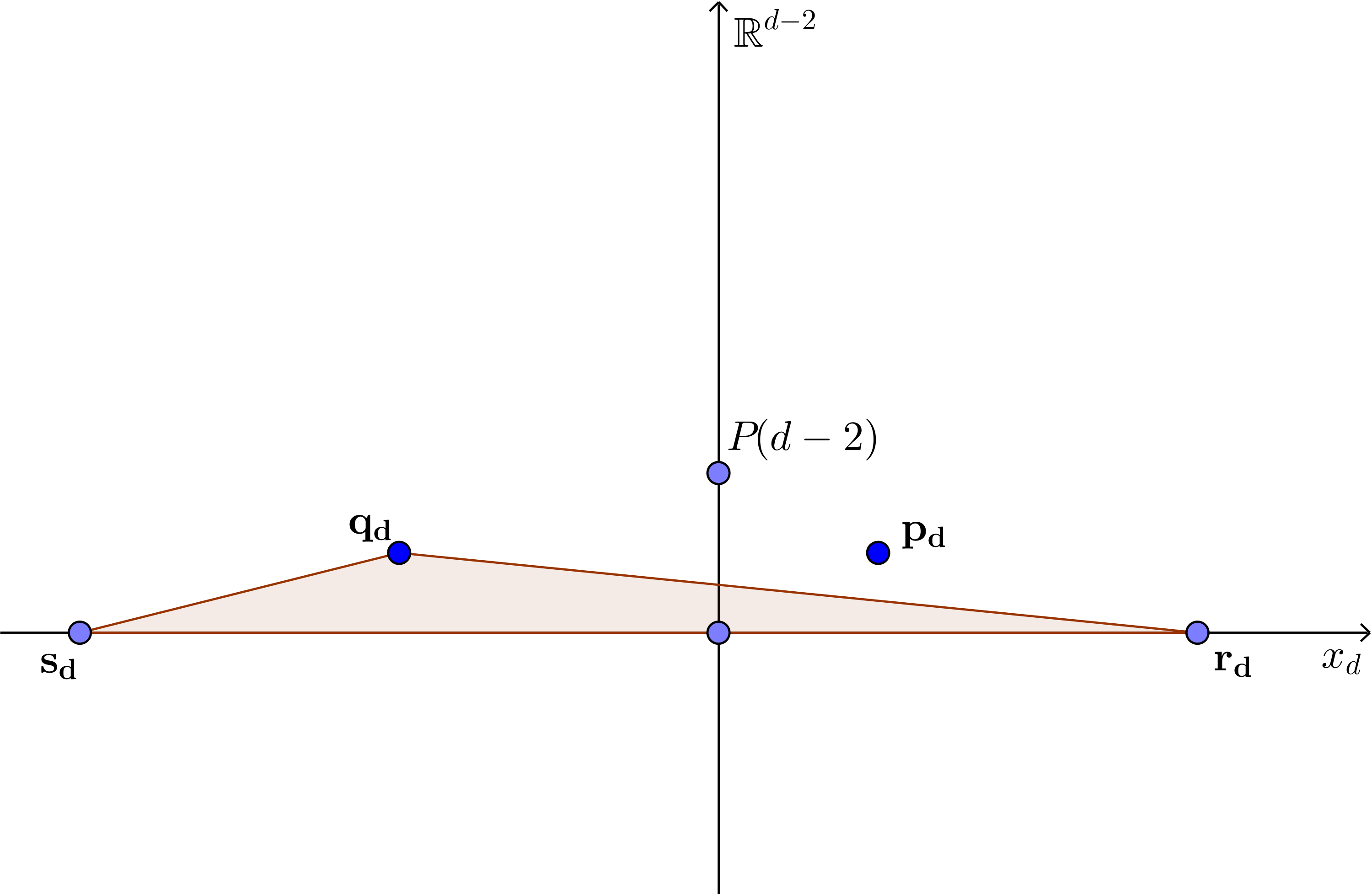}
\end{center}
\caption{The point $\qpoint{d}$ leaves.}\label{jamiethree}
\end{minipage}
\begin{minipage}{0.49\textwidth}
\begin{center}
\includegraphics[scale=0.24]{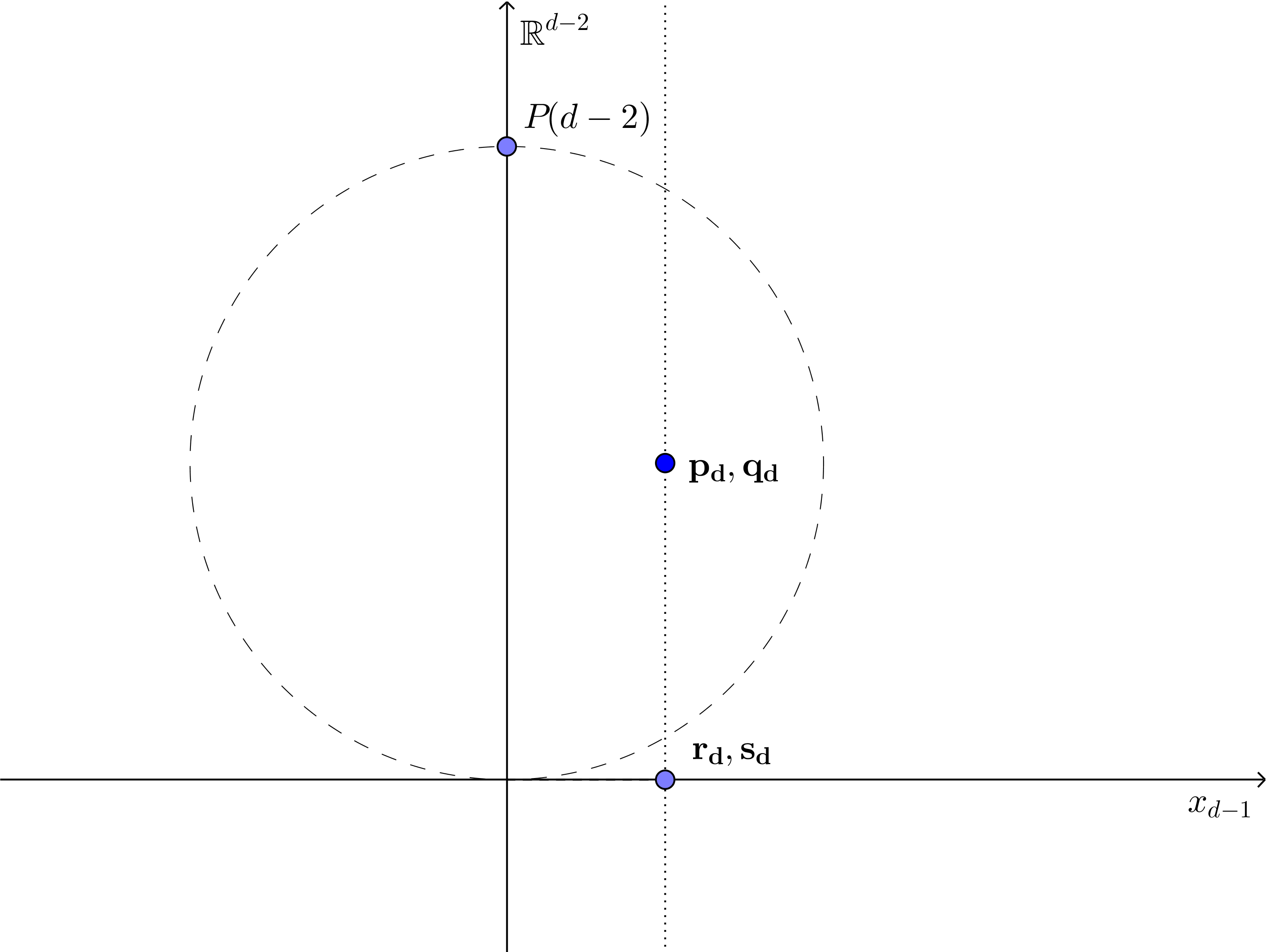}
\end{center}
\caption{The improving points are $\Pset{d-2}$.}\label{jamie3b}
\end{minipage}
\end{figure}

\begin{claim}
The set of improving points is now $\Pset{d-2}$ (with two zero coordinates appended).
\end{claim}
\begin{claimproof}
Now Wolfe's criterion hyperplane contains the four points $\ppoint{d},\qpoint{d},\rpoint{d},\spoint{d}$ by construction leaving $\Pset{d-2}$ on the same
side as the origin (see Figure \ref{jamie3b}).
\end{claimproof}

The first (and minimum norm) point in $\Pset{d}$ enters and the next corral is this point together with $\rpoint{d}$ and $\spoint{d}$.
That is, the next corral is precisely the first corral in $\Cset{d-2} \rpoint{d} \spoint{d}$. 
We will prove inductively that the sequence of corrals from now on is exactly all of $\Cset{d-2} \rpoint{d} \spoint{d}$.
To see this, we repeatedly invoke \cref{lem:secondround} after every corral with $A$ equal to the subspace spanned by the first 
$d-2$ coordinate vectors of $\RR^d$.
Suppose that the current corral is $C \rpoint{d} \spoint{d}$, where $C$ is one of the corrals in $\Cset{d-2}$ and denote the next corral in $\Cset{d-2}$ by $C'$.
From \cref{lem:secondround}, we get that the set of improving points for corral $C \rpoint{d} \spoint{d}$ is obtained by taking the set of improving points 
for corral $C$ and removing $\{\ppoint{d}, \qpoint{d}, \rpoint{d}, \spoint{d}\}$.
Thus, the point that enters is the same that would enter after corral $C$. Let $\ve{a}$ denote that point.
\begin{claim}
The next corral is $C'\rpoint{d}\spoint{d}$.
\end{claim}
\begin{claimproof}
The current set of points is $C\rpoint{d} \spoint{d} \ve{a} $.
If $C \ve{a} $ is a corral then so is $C\rpoint{d} \spoint{d} \ve{a} = C' \rpoint{d} \spoint{d}$ (by \cref{lem:orthogonal}, part 3) and the claim holds.
If $C \ve{a} $ is not a corral, it is enough to prove that the sequence of points removed by the inner loop of Wolfe's method on this set is the same as the sequence on set $C\rpoint{d} \spoint{d} \ve{a}$.
We will show this now by simultaneously analyzing the execution of the inner loop on $C \ve{a} $ and $C\rpoint{d} \spoint{d} \ve{a}$. 
We distinguish the two cases with the following notation: variables are written without a bar ($\bar{\phantom{x}}$) and with a bar, respectively.

Let $\ve{x_1}, \dotsc, \ve{x_k}$ be the sequence of current points constructed by the inner loop on $C \ve{a}$. 
Let $\ve{p_1}, \dotsc, \ve{p_k}$ be the sequence of removed points.
Let $C_1, \dotsc, C_k$ be the sequence of current sets of points at every iteration.
Let $\ve{\bar x_1}, \dotsc, \ve{\bar x_{\bar k}}$ be the corresponding sequence on $C \rpoint{d} \spoint{d} \ve{a}$.
Let $\ve{\bar p_1}, \dotsc, \ve{\bar p_{\bar k}}$ be the corresponding sequence of removed points.
Let $\bar C_1, \dotsc, \bar C_{\bar k}$ be the corresponding sequence of current sets of points.
We will show inductively: $k = \bar k$, there is a one-to-one correspondence between sequences $(\ve{x_i})$ and $(\ve{\bar x_i})$, and $(\ve{p_i}) = (\ve{\bar p_i})$.
More specifically, the correspondence is realized by maintaining the following invariant in the inner loop: $\ve{\bar x_i}$ is a strict convex combination of $\ve{x_i}$ and the minimum norm point in $[\rpoint{d}, \spoint{d}]$.

For the base case, from \cref{lem:orthogonal}, part 2, we have that $\ve{\bar x_1}$ is a strict convex combination of $\ve{x_1}$ (which is the minimum norm point in $\conv C$) and the minimum norm point in segment $[\rpoint{d}, \spoint{d}]$, specifically $\ve{w} := \frac{\m{d-2}}{4}\ve{e_{d-1}}$.

For the inductive step, if $\ve{x_i}$ is a strict convex combination of the current set of points $C_i$, then so is $\ve{\bar x_i}$ of $\bar C_i$ and the inner loop ends in both cases with corrals $C_i = C'$ and $\bar C_i = C' \rpoint{d} \spoint{d}$, respectively.
The claim holds.
If $\ve{x_i}$ is not a strict convex combination of the current set of points $C_i$, then neither is $\ve{\bar x_i}$ of $\bar C_i$.
The inner loop then continues by computing the minimum norm point $\ve{y}$ in $\aff C_i$ and $\ve{\bar y}$ in $\aff \bar C_i$, respectively. 
It then finds point $\ve{z}$ in $\conv C_i$ that is closest to $\ve{y}$ in segment $[\ve{x_i}, \ve{y}]$. 
It finds $\ve{\bar z}$, respectively. 
It then selects a point $\ve{p_i}$ to be removed, and a point $\ve{\bar p_i}$, respectively.
From \cref{lem:orthogonal}, part 2, we have that $\ve{\bar y}$ is a strict convex combination of $\ve{y}$ and $\ve{w}$.

We will argue that $\ve{\bar z}$ is a strict convex combination of $\ve{z}$ and $\ve{w}$.
To see this, we note that segment $[\ve{\bar x_i}, \ve{\bar y}]$ lies in the hyperplane where the last coordinate is 0.
Therefore we only need to intersect it with the part of $\conv \bar C_i$ that lies in that hyperplane.
This part is exactly $\conv (C_i \cup \{\ve{w}\})$, which can be written in a more explicit way as the union of all segments of the form $[\ve{b},\ve{w}]$ with $\ve{b} \in C_i$.
Even more, we only need to look at triangle $\ve{w}, \ve{x_i}, \ve{y}$, as all relevant segments lie on it.
The intersection of this triangle with $\conv C_i$ is segment $[\ve{x_i}, \ve{z}]$ and therefore the intersection of the triangle with $\conv \bar C_i$ is simply triangle $\ve{x_i}, \ve{z}, \ve{w}$.
This implies that the intersection between segment $[\ve{\bar x_i}, \ve{\bar y}]$ and $\conv \bar C_i$ is the same as the intersection between  segment $[\ve{\bar x_i}, \ve{\bar y}]$ and triangle $\ve{x_i}, \ve{z}, \ve{w}$. 
This intersection is an interval $[\ve{\bar x_i}, \ve{\bar z}]$ where $\ve{\bar z}$ is a strict convex combination of $\ve{w}$ and $\ve{z}$ and $\ve{\bar z}$ is the closest point to $\ve{\bar y}$ in that intersection.

It follows that the set of potential points to be removed is the same for the two executions. 
Specifically, if $\ve{z}$ is a strict convex combination of a certain subset $C^*$ of $C_i$, then $\ve{\bar z}$ is a strict convex combination of $C^* \cup \{\rpoint{d}, \spoint{d}\}$.
The sets of points that can potentially be removed are $C_i \setminus C^*$ and $\bar C_i \setminus (C^* \cup \{\rpoint{d}, \spoint{d}\}) = C_i \setminus C^*$ (the same), respectively.
In particular\footnote{Under a mild consistency assumption on the way a point is chosen when there is more than one choice, for example, ``choose the point with smallest index among potential points.''}, $\ve{p_i} = \ve{\bar p_i}$.
This implies $C_{i+1} = \bar C_{i+1}$.
Also, $\ve{x_{i+1}} = \ve{z}$ and $\ve{\bar x_{i+1}} = \ve{\bar z}$ is in $[\ve{x_{i+1}}, \ve{w}]$.
This completes the inductive argument about the inner loop and proves the claim.
\end{claimproof}
This completes the proof of \cref{thm:lowerbound}.
\end{proof}




\section{Linear optimization reduces to minimum-norm problems on simplices}

We reduce linear programming to the minimum norm point problem over a simplex via a series of strongly polynomial time reductions.  
The algorithmic problems we will consider are defined below.  
\ifnum\version=\stocversion
\else
We give definitions for the problems of linear programming (LP), feasibility (FP), bounded feasibility (BFP), V-polytope membership (VPM), zero V-polytope membership (ZVPM), zero V-polytope membership decision (ZVPMD), distance to a V-polytope (DVP) and distance to a V-simplex (DVS). (Prefix ``V-" means that the respective object is specified as the convex hull of a set of points.) See \cite{schrijver98, lovasz1988geometric, MR1956924} for a detailed discussions of strongly polynomial time algorithms.
\fi

\begin{definition}
Consider the following computational problems:
\begin{itemize}
\item 
\textbf{\textup{LP:}} 
Given a rational matrix $A$, a rational column vector $\ve{b}$, and a rational row vector $\ve{c}^T$, output rational $\ve{x} \in \text{argmax}\{\ve{c}^T\ve{x} \suchthat A\ve{x} \le \ve{b}\}$ if $\max\{\ve{c}^T\ve{x} \suchthat A\ve{x} \le \ve{b}\}$ is finite, otherwise output INFEASIBLE if $\{\ve{x} \suchthat A\ve{x} \le \ve{b}\}$ is empty and else output INFINITE.

\ifnum\version=\stocversion
\else

\item 
\textbf{\textup{FP:}} 
Given a rational matrix $A$ and a rational vector $\ve{b}$, if $P:= \{\ve{x} \suchthat A\ve{x} = \ve{b}, \ve{x} \ge \ve{0}\}$ is nonempty, output a rational $\ve{x} \in P$, otherwise output NO.

\item
\textbf{\textup{BFP:}} 
Given a rational $d \times n$ matrix $A$, a rational vector $\ve{b}$ and a rational value $M > 0$, if $P:= \{\ve{x} \suchthat A\ve{x} = \ve{b}, \ve{x} \ge \ve{0}, \sum_{i=1}^n x_i \le M\}$ is nonempty, output a rational $\ve{x} \in P$, otherwise output NO.

\item
\textbf{\textup{VPM:}} 
Given a rational $d \times n$ matrix $A$ and a rational vector $\ve{b}$, if $P:=\{\ve{x} \suchthat A\ve{x}=\ve{b}, \ve{x} \ge \ve{0}, \sum_{i=1}^n x_i = 1\}$ is nonempty, output a rational $\ve{x} \in P$, otherwise output NO.

\item
\textbf{\textup{ZVPM:}} 
Given a rational $d \times n$ matrix $A$, if $P:=\{\ve{x} \suchthat A\ve{x} = \ve{0}, \ve{x} \ge \ve{0}, \sum_{i=1}^n x_i = 1\}$ is nonempty, output a rational $\ve{x} \in P$, otherwise output NO.

\item
\textbf{\textup{ZVPMD:}} 
Given rational points $\ve{p_1}, \ve{p_2}, \dotsc,\ve{p_n} \in \mathbb{R}^d$, output YES if $\ve{0} \in \conv\{\ve{p_1},\ve{p_2},...,\ve{p_n}\}$ and NO otherwise.

\item
\textbf{\textup{DVP:}} 
Given rational points $\ve{p_1}, \ve{p_2}, \dotsc,\ve{p_n} \in \mathbb{R}^d$ defining $P = \conv\{\ve{p_1}, \ve{p_2}, ..., \ve{p_n}\}$, output $d(\ve{0},P)^2$.\details{It is not assumed that the points are all vertices.}

\fi 

\item
\textbf{\textup{DVS:}} 
Given $n \leq d+1$ affinely independent rational points $\ve{p_1}, \ve{p_2},...,\ve{p_{n}} \in \mathbb{R}^d$ defining $(n-1)$-dimensional simplex $P = \conv\{\ve{p_1}, \ve{p_2}, ..., \ve{p_{n}}\}$, output $d(\ve{0},P)^2$.
\end{itemize}
\end{definition}

The main result in this section reduces linear programming to finding the minimum norm point in a (vertex-representation) simplex.
\begin{theorem}
LP reduces to DVS in strongly-polynomial time.
\end{theorem}
\ifnum\version=\stocversion
\else
To prove each of the lemmas below, we illustrate the problem transformation and its strong polynomiality.  
The first two reductions are highly classical, while those following are intuitive, but we do not believe have been written elsewhere. 
\ifnum\version=\stocversion
\else
\fi

Below is the sequence of algorithmic reductions that reduce LP to DVS.

\begin{lemma}\label{lem:LPtoFP}
LP reduces in strongly-polynomial time to FP.
\end{lemma}

\begin{proof}
Let $\oracle$ denote the FP oracle. 
\begin{algorithmic}
\Require $A \in \QQ^{d \times n}, \ve{b} \in \QQ^{d}, \ve{c} \in \QQ^{n}$.
\State Invoke $\oracle$ on 
\begin{equation}\label{eq:LPfeas}
\begin{bmatrix} A & -A & I_d \end{bmatrix} \begin{bmatrix} \ve{x^+} \\ \ve{x^-} \\ \ve{s} \end{bmatrix} = \ve{b}, \begin{bmatrix} \ve{x^+} \\ \ve{x^-} \\ \ve{s} \end{bmatrix} \ge \ve{0}.
\end{equation}
If the output is NO, output INFEASIBLE.

\State Invoke $\oracle$ on 
\begin{equation}\label{eq:LPfinite}
\begin{bmatrix}-\ve{c}^T & \ve{c}^T & \ve{b}^T & \\ A & -A & 0 & I_{d+2n+1} \\ 0 & 0 & A^T & \\ 0 & 0 & -A^T & \end{bmatrix} \begin{bmatrix} \ve{x^+} \\ \ve{x^-} \\ \ve{y} \\ \ve{s} \end{bmatrix} = \begin{bmatrix} 0 \\ \ve{b} \\ \ve{c} \\ -\ve{c} \end{bmatrix}, \begin{bmatrix} \ve{x^+} \\ \ve{x^-} \\ \ve{y} \\ \ve{s} \end{bmatrix} \ge \ve{0}.
\end{equation}
If the output is NO, output INFINITE, else output rational $\ve{x} = \ve{x^+} - \ve{x^-}$.
\end{algorithmic}

\begin{claim}
A solution 
\[
 \ve{\tilde{x}} := \begin{bmatrix} \ve{x^+} \\ \ve{x^-} \\ \ve{s} \end{bmatrix}
\]
to (\ref{eq:LPfeas}) gives a solution to $A\ve{x} \le \ve{b}$ and vice versa. 
\end{claim}
\begin{claimproof}
Suppose $\ve{\tilde{x}}$ satisfies (\ref{eq:LPfeas}).  
Then $A\ve{x^+} - A\ve{x^-} + \ve{s} = \ve{b}$.  Define $\ve{x} = \ve{x^+} - \ve{x^-}$ and note $\ve{s} \ge \ve{0}$.  
Then $A\ve{x} \le \ve{b}$.  
Now, suppose $\ve{x}$ satisfies $A\ve{x} \le \ve{b}$.  
Let $\ve{x^+}$ be the positive coordinates of the vector $\ve{x}$ and $\ve{x^-}$ be the negative components in absolute value, so $x_i^+ = \max(x_i,0)$ and $x_i^- = \max(-x_i,0)$.  
Define $\ve{s} = \ve{b}- A\ve{x}$.  Since $A\ve{x} \le \ve{b}$, we have that $\ve{s} \ge \ve{0}$ and by construction, $\ve{x^+}, \ve{x^-} \ge \ve{0}$. 
Note that $\begin{bmatrix} A & -A & I \end{bmatrix} \ve{\tilde{x}} = A\ve{x^+} - A\ve{x^-} + \ve{s} = A(\ve{x^+} - \ve{x^-}) + \ve{b} - A\ve{x} = A\ve{x} + \ve{b} - A\ve{x} = \ve{b}.$
\end{claimproof}

\begin{claim} 
A solution 
\[
\ve{\tilde{z}} := \begin{bmatrix} \ve{x^+} \\ \ve{x^-} \\ \ve{y} \\ \ve{s} \end{bmatrix}
\] 
to (\ref{eq:LPfinite}) gives a solution to $\text{argmax}\{\ve{c}^T\ve{x} | A\ve{x} \le \ve{b}\}$ and vice versa.
\end{claim}

\begin{claimproof}
Suppose $\ve{\tilde{z}}$ is a solution to (\ref{eq:LPfinite}).  
These are the KKT conditions for the LP $\text{argmax}\{\ve{c}^T\ve{x} | A\ve{x} \le \ve{b}\}$, so $\ve{x} = \ve{x^+} - \ve{x^-}$ is the optimum.  
Suppose $\ve{x} \in \text{argmax}\{\ve{c}^T\ve{x} | A\ve{x} \le \ve{b}\}$.  
By strong duality, there exists $\ve{y}$ so that $\ve{b}^T\ve{y} \le \ve{c}^T\ve{x}$ and $A^T\ve{y} = \ve{c}, \ve{y} \ge \ve{0}$.  
Thus, letting $\ve{x^+}$ and $\ve{x^-}$ be as above, we have $$-\ve{c}^T(\ve{x^+} - \ve{x^-}) + \ve{b}^T\ve{y} \le 0,\; A(\ve{x^+} - \ve{x^-}) \le \ve{b},\; A^T\ve{y}\le \ve{c},\; -A^T\ve{y} \le -\ve{c}.$$  
Now choose $\ve{s} \ge \ve{0}$ so that 
\[
\ve{c}^T\ve{x^+} - \ve{c}^T \ve{x^-} + \ve{b}^T \ve{y} + s_1 = 0, \; A\ve{x^+} - A\ve{x^-} + \ve{s_2^{m+1}} = \ve{b},\; A^T\ve{y} + \ve{s_{m+2}^{n+m+1}} = \ve{c},\; -A^T\ve{y} + \ve{s_{n+m+2}^{2n+m+1}} = -\ve{c}
\] 
where $\ve{s_i^j}$ denotes the subvector of $\ve{s}$ of coordinates $s_i, s_{i+1}, ..., s_{j-1}, s_j$.  
Thus, $\ve{\tilde{z}}$ satisfies (\ref{eq:LPfinite}).
\end{claimproof}

Clearly, constructing the required FP problems takes strongly polynomial time and we have only two calls to $\oracle$, so the reduction is strongly-polynomial time.
\end{proof}

\begin{lemma}
FP reduces in strongly-polynomial time to BFP.
\end{lemma}



\begin{proof}
Let $\oracle$ denote the oracle for BFP. Suppose $A = (a_{ij}/\alpha_{ij})_{i,j = 1}^{d,n}$, $\ve{b} = (b_j/\beta_j)_{j=1}^d$ and define $D := \max(\max_{i \in [d], j \in [n]} |\alpha_{ij}|, \max_{k \in [d]} |\beta_k|)$ and $N : =  \max(\max_{i \in [d], j \in [n]} |a_{ij}|, \max_{k \in [d]} |b_k|)+1$.  If the entry of $A$, $a_{ij}/\alpha_{ij} = 0$ or the entry of $\ve{b}$, $b_j/\beta_j = 0$ define $a_{ij} = 0$ and $\alpha_{ij}=1$ or $b_j = 0$ and $\beta_j = 1$.
\begin{algorithmic}
\Require $A \in \QQ^{d \times n}, \ve{b} \in \QQ^{d}$.
\State Invoke $\oracle$ on $A\ve{x} = \ve{b}, \ve{x} \ge \ve{0}, \sum_{i=1}^n x_i \le n D^{d(n+1)\min(d^3,n^3)}N^{d(n+1)}$.  If the output is NO, output NO, else output rational $\ve{x}$.
\end{algorithmic}

\begin{claim}  
The FP $A\ve{x} = \ve{b}, \ve{x} \ge \ve{0}$ is feasible if and only if the BFP $A\ve{x} = \ve{b}, \ve{x} \ge \ve{0}, \sum_{i=1}^n x_i \le n D^{d(n+1)\min(d^3,n^3)}N^{d(n+1)}$ is feasible.
\end{claim}

\begin{claimproof}  
If the BFP is feasible then clearly the FP is feasible.  Suppose the FP is feasible.  
By the theory of minimal faces of polyhedra, we can reduce this to a FP defined by a square matrix, $A$, in the following way:  
By \cite[Theorem 1.1]{chernikov1965convolution}, there is a solution, $\ve x$, with no more than $\min(d,n)$ positive entries so that $A\ve{x} = \ve{b}$ and the positive entries of $\ve{x}$ combine linearly independent columns of $A$ to form $\ve{b}$.  
Let $A'$ denote the matrix containing only these linearly independent columns and $\ve{x}'$ denote only the positive entries of $\ve{x}$.  Then $A' \ve{x}' = \ve{b}$.  Now, note that $A' \in \QQ^{d \times m}$ where $m \le d$.  Since the column rank of $A'$ equals the row rank of $A'$, we may remove $d-m$ linearly dependent rows of $A'$ and the corresponding entries of $\ve{b}$, forming $A''$ and $\ve{b}'$ so that $A'' \ve{x}' = \ve{b}'$ where $A'' \in \QQ^{m \times m}$, $\ve{b}' \in \QQ^{m}$ and $A''$ is a full-rank matrix.

Define $M:= \prod_{i,j=1}^m |\alpha_{i,j}''| \prod_{k=1}^m |\beta_{k}'|$ and note that $M \le D^{d(n+1)}$.  
Define $L := \prod_{i,j=1}^m (|a_{i,j}''|+1) \prod_{k=1}^m (|b_{k}'|+1)$ and note that $L \le N^{d(n+1)}$.  
Define $\bar{A} = MA''$ and $\bar{\ve{b}} = Mb'$ and note that $\bar{A}$ and $\bar{\ve{b}}$ are integral.  By Cramer's rule, we known that $x_i' = \frac{|\text{det} \bar{A}_i|}{|\text{det} \bar{A}|}$ where $\bar{A}_i$ denotes $\bar{A}$ with the $i$th column replaced by $\bar{\ve{b}}$.  By integrality, $|\text{det} \bar{A}| \ge 1$, so $x_i' \le |\text{det} \bar{A}_i| \le \prod_{i,j=1}^m M (|a_{ij}|+1) \prod_{k=1}^m M (|b_k|+1) = M^{m^3}L \le D^{d(n+1)\min(d^3,n^3)}N^{d(n+1)}$.  Now, note that $\ve{x}'$ defines a solution, $\ve{x}$, to the original system of equations.  Let $x_i = x_j'$ if the $j$th column of $A'$ was the selected $i$th column of $A$  and $x_i=0$ if the $i$th column of $A$ was not selected.  Note then that $A \ve{x} = \ve{b}, \ve{x} \ge \ve{0}, \sum_{i=1}^n x_i \le n D^{d(n+1)\min(d^3,n^3)}N^{d(n+1)}$.
\end{claimproof}

Thus, we have that the FP and BFP are equivalent.  To see that this is a strongly-polynomial time reduction, note that adding this additional constraint takes time for constructing the number $n D^{d(n+1)\min(d^3,n^3)}N^{d(n+1)}$ plus small constant time.  This number takes $d(n+1)$ comparisons and $d(n+1)\min(d^3,n^3)$ multiplications to form.  Additionally, this number takes space which is polynomial in the size of the input (polynomial in $d$,$n$ and size of $D$, $N$).  
\end{proof}

\begin{lemma} BFP reduces in strongly-polynomial time to VPM.
\end{lemma}

\begin{proof}
Let $\oracle$ denote the oracle for VPM.
\begin{algorithmic}
\Require $A \in \QQ^{d \times n}, b \in \QQ^{d}, 0 < M \in \QQ$.
\State Invoke $\oracle$ on 
\begin{equation}\label{eq:VPM}
    \begin{bmatrix} MA & 0 \end{bmatrix} \begin{bmatrix} \ve{y} \\ z \end{bmatrix} = \ve{b}, \begin{bmatrix} \ve{y} \\ z \end{bmatrix} \ge \ve{0}, z + \sum_{i=1}^n y_i= 1.
\end{equation}  
If the output is NO, output NO, else output rational $\ve{x} = M\ve{y}$.
\end{algorithmic}

\begin{claim}
A solution 
$$
\ve{\tilde{w}} := \begin{bmatrix} \ve{y} \\ z \end{bmatrix}
$$ 
to (\ref{eq:VPM}) gives a solution the BFP instance, $A\ve{x}=\ve{b}, \ve{x} \ge \ve{0}, \sum_{i=1}^n x_i \le M$ and vice versa.
\end{claim}

\begin{claimproof} 
Suppose $\ve{\tilde{w}}$ satisfies (\ref{eq:VPM}).  
Then $\ve{x} = M\ve{y}$ is a solution to the BFP instance since $A\ve{x} = MA\ve{y} = \ve{b}$ and since $\ve{y} \ge \ve{0}$, $\ve{x} = M\ve{y} \ge \ve{0}$ and since $\sum_{i=1}^n y_i + z = 1$, we have $\sum_{i=1}^n y_i \le 1$ so $\sum_{i=1}^n x_i = M \sum_{i=1}^n y_i \le M$.  
Suppose $\ve{x}$ is a solution to the BFP instance.  
Then $\ve{y} = \frac{1}{M}\ve{x}$ and $z = 1 - \sum_{i=1}^n y_i$ satisfies (\ref{eq:VPM}), since $\begin{bmatrix} MA & 0 \end{bmatrix} \ve{\tilde{w}} = MA\ve{y} = A\ve{x} = \ve{b}$, $\ve{y} \ge \ve{0}$ since $\ve{x} \ge \ve{0}$ and since $\sum_{i=1}^n x_i \le M$, we have $\sum_{i=1}^n y_i = \frac{1}{M} \sum_{i=1}^n x_i \le 1$ so $z \ge 0$.
\end{claimproof}

Clearly, this reduction is simply a rewriting, so the reduction is strongly-polynomial time.
\end{proof}

\begin{lemma} VPM reduces in strongly-polynomial time to ZVPM.
\end{lemma}

\begin{proof}
Let $\oracle$ be the oracle for ZVPM.
\begin{algorithmic}
\Require $A \in \QQ^{d \times n}, b \in \QQ^{d}$.
\State Invoke $\oracle$ on 
\begin{equation}\label{eq:ZVPM}
    \begin{bmatrix} \ve{a_1}-\ve{b} & \ve{a_2} - \ve{b} & \cdots & \ve{a_n} - \ve{b} \end{bmatrix} \ve{x} = \ve{0}, \ve{x} \ge \ve{0}, \sum_{i=1}^n x_i = 1
\end{equation} where $\ve{a_i} \in \QQ^m$ is the $i$th column of $A$.  If the output is NO, output NO, else output rational $\ve{x}$.
\end{algorithmic}

\begin{claim} 
A solution to (\ref{eq:ZVPM}) gives a solution to the VPM instance and vice versa.
\end{claim}

\begin{claimproof}
Note that $\ve{x}$ satisfies (\ref{eq:ZVPM}) if and only if $0 = \sum_{i=1}^n x_i(\ve{a_i} -\ve{b}) = \sum_{i=1}^n x_i \ve{a_i} - \ve{b} \sum_{i=1}^n x_i = A\ve{x} - \ve{b}$ so $A\ve{x} = \ve{b}$.  
Thus, $\ve{x}$ is a solution to the VPM instance if and only if $\ve{x}$ is a solution to (\ref{eq:ZVPM}). 
\end{claimproof}

Clearly, this reduction is simply a rewriting, so the reduction is strongly-polynomial time.
\end{proof}






\begin{lemma}
ZVPM reduces in strongly-polynomial time to ZVPMD.
\end{lemma}
\begin{proofidea}
The reduction sequentially asks for every vertex whether it is redundant and if so, it removes it and continues. 
This process ends with at most $d+1$ vertices so that $\ve{x}$ is a strict convex combination of them and the coefficients $x_i$ can be found in this resulting case by solving a linear system.
\end{proofidea}

\begin{proof}
Let $\oracle$ denote the ZVPMD oracle.
\begin{algorithmic}
\Require $P:= \{\ve{A_1}, \dotsc, \ve{A_n}\} \subseteq \QQ^d$ where $A_i$ is the $i$th column of $A$.
\State Invoke $\oracle$ on $P$. If the output is NO, output NO.
\For{$i=1, \dotsc, n$}
\State Invoke $\oracle$ on instance $P$ without $\ve{A_i}$. If output is YES, remove $\ve{A_i}$ from $P$.
\EndFor
\State Let $m$ be the cardinality of $P$.
\State Output the solution $x_1, \dotsc, x_m$ to the linear system $\sum x_i = 1$, $\sum_{\ve{p_i} \in P} x_i \ve{p_i} = \ve{0}$
\end{algorithmic}
Let $P^*$ be the resulting set of points $P$ after the loop in the reduction. Claim: $P^*$ contains at most $d+1$ points so that $\ve{0}$ is a strict convex combination of (all of) them.
Proof of claim: 
By Caratheodory's theorem there is a subset $Q \subseteq P^*$ of at most $d+1$ points so that $\ve{0}$ is a strict convex combination of points in $Q$. 
We will see that $P^*$ is actually equal to $Q$. 
Suppose not, for a contradiction. Let $\ve{p} \in P^* \setminus Q$. 
At the time the loop in the reduction examines $\ve{p}$, no point in $Q$ has been removed and therefore $\ve{p}$ is redundant and is removed. 
This is a contradiction.
\end{proof}

\ifnum\version=\stocversion
\else
In our next lemma, we make use of the following elementary fact.
\begin{claim}\label{claim:affine}
Given $A$ an $m \times n$ matrix let $B$ be $A$ with a row of $1$s appended. 
The columns of $A$ are affinely independent if and only if the columns of $B$ are linearly independent. 
The convex hull of the columns of $A$ is full dimensional if and only if rank of $B$ is $m+1$.
\end{claim}
\fi


\begin{lemma}
ZVPMD reduces in strongly-polynomial time to DVS.
\end{lemma}
\begin{proof}
Clearly ZVPMD reduces in strongly-polynomial time to DVP: Output YES if the distance is 0, output NO otherwise.


Given an instance of distance to a V-polytope, $\ve{p_1}, \ve{p_2}, \dotsc, \ve{p_{n}}$, we reduce it to an instance of DVS as follows:
We lift the points to an affinely independent set in higher dimension, a simplex, by adding small-valued new coordinates.
\Cref{claim:affine} allows us to handle affine independence in matrix form.
Let $A$ be the $d \times n$ matrix having columns $(\ve{p_i})_{i=1}^n$. Let $\ve{v_1}, \dotsc, \ve{v_d}$ be the rows of $A$. 
Let $\ve{v_0} \in \RR^n$ be the all-ones vector.
We want to add vectors $\ve{v_{d+1}}, \dotsc, \ve{v_{d+t}}$, for some $t$, so that $\ve{v_0}, \dotsc, \ve{v_{d+t}}$ is of rank $n$.
To this end, we construct an orthogonal basis (but not normalized, to preserve rationality) of the orthogonal complement of $\linspan{\ve{v_0}, \dotsc, \ve{v_d}}$. 
The basis is obtained by applying the Gram-Schmidt orthogonalization procedure (without the normalization step) to the sequence $\ve{v_0}, \dotsc, \ve{v_d}, \ve{e_1}, \dotsc, \ve{e_n}$. 
Denote $\ve{v_{d+1}}, \dotsc, \ve{v_{d+t}}$ the resulting orthogonal basis of the orthogonal complement of $\linspan{\ve{v_0}, \dotsc, \ve{v_d}}$. 
The matrix with rows $\ve{v_0}, \dotsc, \ve{v_d}, \ve{v_{d+1}}, \dotsc, \ve{v_{d+t}}$ is of rank $n$ and so is the matrix with rows
\[
\ve{v_0}, \dotsc, \ve{v_d}, \eps \ve{v_{d+1}}, \dotsc, \eps \ve{v_{d+t}}
\]
for any $\eps > 0$ (to be fixed later). Therefore, the $n$ columns of this matrix are linearly independent.
Let $B$ be the matrix with rows
\[
\ve{v_1}, \dotsc, \ve{v_d}, \eps \ve{v_{d+1}}, \dotsc, \eps \ve{v_{d+t}}.
\]
Let $\ve{w_1}, \dotsc, \ve{w_n}$ be the columns of $B$. 
By construction and \cref{claim:affine} they are affinely independent.
Let $S$ denote the convex hull of these $(n-1)$-dimensional rational points.
Polytope $S$ is a simplex.
Moreover, if $Q := \conv \{ \ve{p_1}, \dotsc, \ve{p_{n}}\}$, then
\[
d(\ve{0},S)^2 
\leq d(\ve{0},Q)^2 + \eps^2 \sum_{i=d+1}^{d+t} \enorms{\ve{v_{i}}} 
\leq d(\ve{0},Q)^2 + \eps^2 n
\]
(where we use that $\enorm{\ve{v_{i}}} \leq 1$, from the Gram-Schmidt construction).

The reduction proceeds as follows:
Let $T$ be the maximum of the absolute values of all numerators and denominators of entries in $(\ve{p_i})_{i=1}^n$ (which can be computed in strongly polynomial time\footnote{Equivalently, without loss of generality we can assume that the input is integral, and then take $C$ to be the maximum of the absolute values of all entries in $(\ve{p_i})_{i=1}^n$, as done in Schrijver's \cite[Section 15.2]{schrijver98}.}).
From \cref{lem:mindist}, we have $d(\ve{0},Q)^2 \geq \frac{1}{d (dT)^{2d}}$ if $\ve{0} \notin Q$.
Compute rational $\eps > 0$ so that $\eps^2 n < \frac{1}{d (dT)^{2d}}$. 
For example, let $\eps := \frac{1}{nd (dT)^d}$.
The reduction queries $d(\ve{0},S)^2$ for $S$ constructed as above and given by the choice of $\eps$ we just made. It then outputs YES if $d(\ve{0},S)^2 < \frac{1}{d (dT)^{2d}}$ and NO otherwise.
%
\end{proof}

\begin{lemma}\label{lem:mindist}
Let $P = \conv \{\ve{p_1}, \dotsc, \ve{p_n} \}$ be a V-polytope with $\ve{p_i} \in \QQ^d$. 
Let $T$ be the maximum of the absolute values of all numerators and denominators of entries in $(\ve{p_i})_{i=1}^n$.
If $\ve 0 \notin P$ then $d(\ve 0, P) \geq \frac{1}{ (dT)^{d} \sqrt{d}}$.
\end{lemma}

\begin{proof}
The claim is clearly true if $P$ is empty.
If $P$ is non-empty, let $\ve{y} = \proj_P(\ve{x})$.
We have that every facet of $P$ can be written as $\ve a^T \ve x \leq k$, where $\ve a (\neq 0)$ is an integral vector, $k$ is an integer and the absolute values of the entries of $\ve a$ as well as $k$ are less than $(dT)^d$ (\cite[Theorem 3.6]{MR625550}).
By assumption at least one these facet inequalities is violated by 0. 
Denote by $\ve a^T \ve x \leq k$ one such inequality.
Let $H = \{\ve{x} \suchthat \ve a^T \ve x = k \}$. 
We have $\enorm{\ve y} = d(0,P) \geq d(0,H)$, and $d(0,H)^2 = k^2/\enorm{\ve a}^2 \geq \frac{1}{d (dT)^{2d}}$.
The claim follows.
\end{proof} 

\fi 

\section{Conclusions and open questions}

We have seen that Wolfe's method using a natural point insertion rule exhibits exponential behavior. We have also shown
that the minimum norm point problem for simplices is intimately related to the complexity of linear programming. Our work raises several very natural questions:

\begin{itemize}
    \item Are there exponential examples for other insertion rules for Wolfe's method? Also, at the moment, the ordering of
    the points starts with the closest point to the origin, but one could also consider a randomized initial rule or a randomized insertion rule.
    
    \item For applications in submodular function minimization, the polytopes one considers are base polytopes and our
    exponential example is not of this kind. Could there be hope that for base polytopes Wolfe's method performs better?
    
    \item It would be interesting to understand the average performance of Wolfe's method. 
    How does it behave for random data? Further randomized analysis of this method would include the smoothed analysis of Wolfe's method or at least the behavior for data following a prescribed distribution. 
    
    \item We have seen that it is already quite interesting to study the minimum norm point problem for simplices, when we discussed the connection with linear programming. Is there a family of simplices where Wolfe's method takes exponential time? 
    
    \item Can Wolfe's method be extended to other convex $L_p$ norms for $p>1$?
\end{itemize}

\subsection*{Acknowledgements}
We thank Gilberto Calvillo, Deeparnab Chakrabarty, Antoine Deza, Stephanie Jegelka, Matthias K\"oppe, Tamon Stephen, and John Sullivan for useful suggestions, conversations and comments about this project. This work was partially supported by NSF grant DMS-1440140, while the first and second authors were in residence at the Mathematical Sciences Research Institute in Berkeley, California, during the Fall 2017 semester. The first and second author were also partially supported by NSF grant DMS-1522158. The second author was also partially supported by the University of California, Davis Dissertation Fellowship. The third author was also partially supported by NSF grants CCF-1657939 and CCF-1422830 while the third author was in residence at the Simons Institute for the Theory of Computing. 

\ifnum\version=\stocversion
    \newpage
\fi

\bibliographystyle{plain}
\bibliography{bib}

\end{document}